\documentclass{amsart}
\input epsf

\usepackage{amsfonts,amsthm,amsmath,amssymb,latexsym}
\usepackage{graphicx,color}
\usepackage[all]{xy}

\begin{document}

\author{Dragomir \v Sari\' c}
\thanks{This research is partially supported by National Science Foundation grant DMS 1102440.}

\address{Department of Mathematics, Queens College of CUNY,
65-30 Kissena Blvd., Flushing, NY 11367}
\email{Dragomir.Saric@qc.cuny.edu}

\address{Mathematics PhD. Program, The CUNY Graduate Center, 365 Fifth Avenue, New York, NY 10016-4309}

\theoremstyle{definition}

 \newtheorem{definition}{Definition}[section]
 \newtheorem{remark}[definition]{Remark}
 \newtheorem{example}[definition]{Example}

\newtheorem*{notation}{Notation}

\theoremstyle{plain}

 \newtheorem{proposition}[definition]{Proposition}
 \newtheorem{theorem}[definition]{Theorem}
 \newtheorem{corollary}[definition]{Corollary}
 \newtheorem{lemma}[definition]{Lemma}

\def\H{{\mathbb H}}
\def\F{{\mathcal F}}
\def\R{{\mathbb R}}
\def\Q{{\mathbb Q}}
\def\Z{{\mathbb Z}}
\def\E{{\mathcal E}}
\def\N{{\mathbb N}}
\def\X{{\mathcal X}}
\def\Y{{\mathcal Y}}
\def\C{{\mathbb C}}
\def\D{{\mathbb D}}
\def\G{{\mathcal G}}

\title{Bendings by finitely additive transverse cocycles}


\date{\today}

\maketitle

\begin{abstract}
Let $S$ be any closed hyperbolic surface and let $\lambda$ be a maximal geodesic lamination on $S$. The amount of bending of an abstract pleated surface (homeomorphic to $S$) with the pleating locus $\lambda$ is completely determined by an $(\mathbb{R}/2\pi\mathbb{Z})$-valued finitely additive transverse cocycle $\beta$ to the geodesic lamination $\lambda$. We give a sufficient condition on $\beta$ such that the corresponding pleating map $\tilde{f}_{\beta}:\mathbb{H}^2\to\mathbb{H}^3$ induces a quasiFuchsian representation of the surface group $\pi_1(S)$. Our condition is genus independent. 
\end{abstract}

\section{Introduction}

Let $S$ be a closed hyperbolic surface and let $\lambda$ be a maximal geodesic lamination on $S$.  
The universal covering $\tilde{S}$ of $S$ is isometrically identified with the hyperbolic plane $\mathbb{H}^2$ . Denote by $\tilde{\lambda}$ the lift of $\lambda$ to $\mathbb{H}^2$ and denote by $\mathbb{H}^3$ the hyperbolic three space. Each component of $\mathbb{H}^2-\tilde{\lambda}$, called {\it plaque}, is an ideal hyperbolic triangle.
A {\it pleated surface} with the {\it pleating locus} $\lambda$  is an immersion $\tilde{f}:\mathbb{H}^2\to \mathbb{H}^3$ which conjugates the covering group of $\mathbb{H}^2\to S$ into a subgroup of the isometries of $\mathbb{H}^3$, and is totally geodesic on plaques and on geodesics of $\tilde{\lambda}$. The pleating map $\tilde{f}$ is an isometry from $\mathbb{H}^2$ onto its image $f(\mathbb{H}^2)$ for the path metric on $f(\mathbb{H}^2)$ induced by the hyperbolic metric of $\mathbb{H}^3$ . The pleatings along $\tilde{\lambda}$ give rise to a finitely additive $(\mathbb{R}/2\pi\mathbb{Z})$-valued transverse cocycle to $\lambda$, and conversely any such cocycle defines a pleated surface with the pleating locus $\lambda$ (cf. Bonahon \cite{Bon2}). The main result in the paper gives a sufficient condition on transverse cocycle such that the pleating map is injective on the boundary.

\subsection{Pleating maps along measured laminations}
Let $\mu$ be a measured (geodesic) lamination on the hyperbolic plane $\mathbb{H}^2$. By definition, $\mu$ is a collection of positive Borel measures on hyperbolic arcs transverse to the support geodesic lamination $|\mu |$. The collection of measures is invariant under homotopies relative the geodesics in the support $|\mu |$. For example, $\mu$ could be the lift of a measured lamination on a closed hyperbolic surface.  

Given a closed geodesic arc $I$ transverse to $|\mu |$,  denote by $\mu (I)\geq 0$ the total mass of the measure deposited on $I$. 
The Thurston norm of $\mu$ is given by
$$
\|\mu\| =\sup_I\mu (I) 
$$
where the supremum is over all closed hyperbolic arcs of length $1$.

Given a complex number $t\in\mathbb{C}$ there exists a (unique) pleated surface corresponding to the complex measure $t\mu$ with the bending along $|\mu |$ determined by the imaginary part of $t\mu$ (and the path hyperbolic metric on the immersed surface determined by the real part of $t\mu$).  
Then we have the following sufficient condition for the pleating map to be an embedding (cf. \cite{EpMarMar}, \cite{Sa1}).

\vskip .2 cm

\paragraph{\bf Theorem.}
{\it Let $\mu$ be a measured lamination on the hyperbolic plane $\mathbb{H}^2$ with finite Thurston norm.
There exist universal $C>0$ and $\epsilon >0$ such that whenever $t\in\mathbb{C}$ satisfies
$$
|Im (t)|< \frac{\epsilon}{ \|\mu\| e^{C\|\mu\|\cdot |Re(t)|}}
$$
the pleating map induced by $t\mu$ is an embedding.}

\vskip .2 cm

The purpose of this paper is to extend the above theorem to finitely additive transverse cocycles to (maximal) geodesic laminations on closed hyperbolic surfaces. Natural examples of finitely additive real valued transverse cocycles arise from real Fenchel-Nielsen coordinates on pants decomposition of closed surfaces. Namely, a  geodesic pairs of pants decomposition of a closed hyperbolic surface can be completed to a maximal geodesic lamination by adding three geodesics to each pair of pants that spiral around cuffs. The obtained maximal geodesic lamination has finitely many  leaves (i.e. geodesics) and the components of the complement of the geodesic lamination are ideal hyperbolic triangles. 
The Fenchel-Nielsen coordinates describe the shape of the pairs of pants (i.e. the cuff lengths) and how two pairs of pants fit together (i.e. the twist parameters) which completely describes the hyperbolic metric on the surface.
On the other hand, a
finitely additive real valued transverse cocycle to the maximal geodesic lamination describes how complementary  ideal triangles fit together to give the hyperbolic metric on the surface (cf. \cite{Bon2}). 

In their proof of the surface subgroup conjecture, Kahn and Markovic \cite{KahnMark} used a sufficient condition on complex Fenchel-Nielsen coordinates to obtain embedded pleated surfaces. In \cite{Sar2} we obtained a sufficient condition on finitely additive transverse cocycles to finite geodesic laminations that gives embedded pleated surfaces. Our main objective is to extend the scope of this theorem to arbitrary maximal geodesic laminations. 

\subsection{Finitely additive cocycles and pleating maps}
Let $\lambda$ be an arbitrary maximal geodesic lamination on a closed hyperbolic surface $S$. The plaques of $\lambda$ are ideal hyperbolic triangles. A {\it finitely additive real valued transverse cocycle} to $\lambda$ is an assignment of a real valued finitely additive measure to each hyperbolic arc transverse to $\lambda$ that is invariant under homotopies relative the leaves of $\lambda$. A finitely additive transverse cocycle defines H\"older distribution on all  
hyperbolic arcs transverse to $\lambda$, and conversely any transverse H\"older distribution defines a finitely additive transverse cocycle to $\lambda$ (cf. Bonahon \cite{Bon1}).

As in the case of finite geodesic laminations, the hyperbolic metric on a closed surface is completely determined by 
the induced finitely additive transverse cocycle. Moreover, the Teichm\"uller space $T(S)$ of a closed hyperbolic surface $S$ is parametrized by open cone $\mathcal{C}(\lambda )$ inside the space $\mathcal{H}(\lambda ,\mathbb{R})$ of finitely additive real valued transverse cocycles to $\lambda$ (cf. Thurston \cite{Th} and Bonahon \cite{Bon2}). A detailed study of the space $\mathcal{H}(\lambda ,\mathbb{R})$ is carried out by Bonahon (cf. \cite{Bon1}, \cite{Bon3}). 

A {\it pleated surface} with the {\it pleating locus} $\lambda$ is an immersion
$$
\tilde{f}:\mathbb{H}^2\to\mathbb{H}^3
$$
which is an isometry on each complementary region to $\tilde{\lambda}$ and on each geodesic of $\tilde{\lambda}$ such that 
$$
\tilde{f}G\tilde{f}^{-1}=K,
$$
where $G$ is the covering group of $\mathbb{H}^2\to S$ and $K$
is a subgroup of the isometry group of $\mathbb{H}^3$. 
A pleated surface along $\lambda$ induces a $(\mathbb{C}/2\pi i\mathbb{Z})$-valued transverse cocycle to $\lambda$ where the real part parametrizes the induced path metric and the imaginary part parametrizes the amount of bending along $\lambda$. Moreover, the space of all representations of $G$ that realize $\lambda$ is parametrized by an open subspace $\mathcal{C}(\lambda )+i\mathcal{H}(\lambda ,\mathbb{R}/2\pi\mathbb{Z})$ of the space $\mathcal{H}(\lambda ,\mathbb{C}/2\pi i\mathbb{Z})$ (cf. Bonahon \cite{Bon2}).

\subsection{The main result} We fix a hyperbolic metric on $S$ and give a sufficient condition on  ($\mathbb{R}/2\pi \mathbb{Z}$)-valued transverse cocycle $\beta$ to the maximal geodesic lamination $\lambda$ such that the extension to the boundary   $\tilde{f}_{\beta}:\partial_{\infty}\mathbb{H}^2\to\partial_{\infty}\mathbb{H}^3$ of the corresponding pleating map $\tilde{f}_{\beta}:\mathbb{H}^2\to\mathbb{H}^3$ is injective.

Let $\{ h_1,\ldots ,h_m\}$ be a set of closed geodesic arcs on $S$ transversely intersecting $\lambda$ (with endpoints in the plaques of $\lambda$) such that the components of $\lambda -\cup_{i=1}^mh_i$ have finite length. In addition, if $h_i$'s are collapsed to points and homotopic arcs of $\lambda -\cup_{i=1}^mh_i$ are identified, we obtain a trivalent (topological) train track which carries $\lambda$.

A {\it gap} of $h_i$ is a component of $h_i-\lambda$ that does not contain an endpoint of $h_i$.  Given a gap $d$ of $h_i$, denote by $g_{d}^{+}$ and $g_{d}^{-}$ two arcs of $\lambda -\{ h_1,\ldots ,h_m\}$ that pass through endpoints of $d$ oriented in the same direction (with respect to $h_i$).
If $g_{d}^{+}$ and $g_{d}^{-}$ together with the gap $d$ and another (on some arc $h_j$ ) gap form a quadrilateral, then we say that $g_{d}^{+}$ and $g_{d}^{-}$ are {\it parallel}.

Since the topological train track (obtained by collapsing $h_i$'s) is trivalent, it follows that each $h_i$ has exactly one gap $d$ such that $g_{d}^{+}$ and $g_{d}^{-}$ are not parallel in one direction. Choose an arbitrary point in the interior of $d$ to divide $h_i$ into two arcs $h_i',h_i''$ with $h_i=h_i'\cup h_i''$. We form a new set of arcs by taking all $h_i,h_i',h_i''$ and for simplicity of notation denote it by $\{ k_1,\ldots ,k_n\}$, called a {\it set of ties for} $\lambda$. 
 (An equivalent definition for a set of ties is given using ``geometric train tracks'', cf. Bonahon \cite{Bon1} and \S 2). 

Fix a set of ties $\{ k_1,\ldots ,k_n\}$ for $\lambda$.
Two arcs $k_i$ and $k_j$ are {\it paired} if there exists an arc in $\lambda -\{ k_1,\ldots ,k_n\}$ that connects them.
Define
\begin{equation}
\label{eq:max_on_edge}
l^{*}=\max_{i,j}diam(k_{i}\bigcup k_{j})
\end{equation}
and
\begin{equation}
\label{eq:min_on_edge}
l_{*}=\min_{i,j}dist(k_{i}, k_{j}),
\end{equation}
where the maximum and the minimum are over all paired arcs $k_i,k_j$, $diam(k_{i}\bigcup k_{j})$ is the diameter of $k_{i}\bigcup k_{j}$, and $dist(k_{i}, k_{j})$ is the distance between $k_i$ and $k_j$.

Moreover, we define
$$
w^{*}=\max_{1\leq i\leq n}|k_i|
$$
and
$$
w_{*}=\min_{1\leq i\leq n}|k_i|,
$$
where $|k_i|$ is the length of the arc $k_i$.

A set of ties $\{ k_1,\ldots ,k_n\}$ for $\lambda$ is said to be {\it geometric} if each angle of the intersection between an arc in $\{ k_1,\ldots ,k_n\}$ and a  geodesics of $\lambda$ is in the interval $[\pi /4,3\pi /4]$,  and 
\begin{equation}
\label{eq:length_bounded_by_delta}
w^{*}\leq 1/20.
\end{equation}

The above quantities $l^{*}$, $l_{*}$, $w^{*}$ and $w_{*}$ give quantitative information about the hyperbolic metric on $S$ and the position of the geodesic lamination $\lambda$ on $S$. We use this information in order to give a sufficient condition on the bending cocycles such that the bending map is injective.

We define the {\it norm} of $\beta$ for the geometric family of arcs $\{ k_1,\ldots ,k_n\}$ by
\begin{equation}
\label{eq:max_on_k_i}
\|\beta\|_{max}=\max \{ |\beta (k_i)|:1\leq i\leq n\} .
\end{equation}
The norm $\|\beta\|_{max}$ is analogous to the Thurston norm of measured laminations.

\vskip .2 cm

In general, a finitely additive real measure $\beta$ on a closed interval $I$ has infinite ``variation'' in the sense that there exists a sequence $I_n$ of subintervals of $I$ with $|\beta (I_n)|\to\infty$ as $n\to\infty$ (cf. Bonahon \cite{Bon3}). In order to find a sufficient condition for the injectivity of the bending map $\tilde{f}_{\beta}:\partial_{\infty}\mathbb{H}^2\to\partial_{\infty}\mathbb{H}^3$, we introduce (below) a notion of ``variation on large gaps'' on the set of ties. (It is important to note that if $\beta$ does not have small enough variation on large gaps in order to have injectivity of $\tilde{f}_{\beta}$ then, for $|t|>0$ small enough, $t\beta$ does have small enough variation on large gaps such that $\tilde{f}_{t\beta}$ is injective.)

Given $\delta>0$, we introduce the $\delta$-variation of $\beta$ on $k_i$ as follows. 
Give an arbitrary orientation to $k_i$.
Let $\{ d_j: j=1,\ldots ,j_i\}$ be a finite family of gaps of $k_i$ together with the two components of $k_i-\lambda$ containing the endpoints of $k_i$. Define $k_{d_j}$, for $j=1,\ldots ,j_i$, to be the subarc of $k_i$ whose initial point is the initial point of $k_i$ and whose endpoint is a point in $d_j$. 

Then the {\it $\delta$-variation of $\beta$ on} $k_i$ 
is given by
\begin{equation}
\label{eq:k_d_measure_of_variation}
\|\beta\|_{var_\delta ,k_i}=\max_{1\leq j\leq j_i}|\beta (k_{d_j})|,
\end{equation}
where the set $\{ d_j: j=1,\ldots ,j_i\}$ is chosen such that the length of $k_i\setminus\cup_{j=1}^{j_i}{d_j}$ is less than $\delta |k_i|$ ($|k_i|$ denotes the length of $k_i$).

Moreover, the {\it $\delta$-variation of $\beta$ on a geometric family} $\{ k_1,\ldots ,k_n\}$ 
is given by
\begin{equation}
\label{eq:variation_epsilon}
\|\beta\|_{var_{\delta}}=\max_{1\leq i\leq n} \|\beta\|_{var_\delta ,k_i}.
\end{equation}

Our main result is a sufficient condition for the injectivity of the bending map corresponding to a transverse cocycle $\beta$ in terms of a geometric set of arcs.

\begin{theorem}
\label{thm:main} There exist  ${\epsilon}  >0$ and $\delta >0$ such that for any closed hyperbolic surface $S$ and a maximal geodesic lamination  $\lambda$ on $S$ the following holds.
Let $\{ k_1,\ldots ,k_n\}$ be a geometric set of arcs for $\lambda$ such that 
\begin{equation}
\label{eq:dist_size_k} w^{*}<\frac{e^{-2l^{*}}\tanh\frac{l_{*}}{2}}{8\pi}.\end{equation} If an $(\mathbb{R}/2\pi\mathbb{Z})$-valued transverse
cocycle $\beta$ to $\lambda$ satisfies
\begin{equation}
\label{eq:max_bdd_by_epsilon}
\|\beta\|_{max}<\epsilon w_{*}
\end{equation}
and 
\begin{equation}
\label{eq:varr_epsilon_bounded_by_epsilon}
\|\beta\|_{var_{\delta}}<\epsilon
\end{equation}
then the developing map $\tilde{f}_{\beta}:\mathbb{H}^2\to\mathbb{H}^3$ continuously extends to an injective map $\tilde{f}_{\beta}:\partial_{\infty}\mathbb{H}^2\to\partial_{\infty}\mathbb{H}^3$.
\end{theorem}

\begin{remark}
The condition (\ref{eq:max_bdd_by_epsilon}) is the standard condition that works for the transverse measures (cf. \cite{EpMarMar}, \cite{Sa1}). The condition (\ref{eq:varr_epsilon_bounded_by_epsilon}) is a new condition needed to control the geometry of the realization of the transverse cocycle $\beta$ due to the fact that the  variation of $\beta$ is unbounded. Note that the condition (\ref{eq:max_bdd_by_epsilon}) for a choice of a family of geometric arcs does not imply similar condition on an arbitrary arc of length at most $1$. For this reason it is necessary that the arcs $k_i$'s are on a relatively large distance compared to their sizes which is made explicit by (\ref{eq:dist_size_k}).
\end{remark}

\begin{remark}
The quantities $l^{*}$ and $l_{*}$ depend on the lamination $\lambda$. The constants $\epsilon$ and $\delta$ are computed in terms of $l^{*}$ and $l_{*}$ in the proof of the theorem. If a geometric set of arcs $\{ k_1,\ldots ,k_n\}$ does not satisfy (\ref{eq:dist_size_k}) then we can divide each arc into several subarcs until the condition is satisfied. If $\lambda$ contains short closed geodesics then $l_{*}$ is small for any choice of a geometric set of arcs for $\lambda$. A generic geodesic lamination $\lambda$ contains no closed geodesics and the choice of a geometric set of arcs can be made such that $l^{*}\geq 1/5$ and $l_{*}\geq l^{*}/4$ in which case we can choose $w^{*}=4.41719\times 10^{-10}$ and
$\epsilon = \delta=3.61749\times
10^{-17}.$ We give a table of values for $\epsilon$ and $\delta$ when $l_{*}=l^{*}/4$ for various values of $l^{*}$ (cf. Table \ref{tab:consts}). It seems that the optimal value is $l^{*}=0.0238523$ in which case $\epsilon =\delta = 2.01795\times 10^{-13}$ and $w^{*}\leq 1.27126\times 10^{-11}$. 
\end{remark}

Let $\alpha$ be an $\mathbb{R}$-valued transverse cocycle to $\lambda$ which is induced by the hyperbolic metric on $S$ (cf \cite{Bon2}).  For $z\in\mathbb{C}$, define the transverse cocycle $\alpha_z$ by 
$$
\alpha_z(k)=(1+z)\alpha (k)\mod (2\pi i\mathbb{Z})
$$
for each arc $k$ transverse to $\lambda$.

The developing shear-bend map $\tilde{f}_{z}:\mathbb{H}^2\to\mathbb{H}^3$ (normalized to be the identity on a fixed plaque of $\lambda$)  corresponding
to the transverse cocycle $\alpha_z$ induces a holomorphic family (in $z$) of
representation of $\pi_1(S)$ to $PSL_2(\mathbb{C})$ (cf. \cite{Bon2}). As a corollary to the above theorem, we obtain

\begin{corollary}
Let $\alpha$ be an $\mathbb{R}$-valued transverse cocycle to a geodesic lamination $\lambda$ corresponding to a hyperbolic metric on a closed surface $S$ and let $\tilde{f}_{z}$ be the
shear-bend map for $\alpha_{z}$. Then there exists $\epsilon >0$ such that
the shear-bend map
$$\tilde{f}_{z}:\mathbb{H}^2\to\mathbb{H}^3$$
extends by continuity to a holomorphic motion of
$\partial_{\infty}\mathbb{H}^2$ in $\partial_{\infty}\mathbb{H}^3$
for the parameter $\{z\in\mathbb{C}:|z|<\epsilon \}$.
\end{corollary}

\subsection{Outline of the proof of main theorem} It is enough to prove that any two points $x,y\in\partial_{\infty}\mathbb{H}^2$ are mapped onto distinct points in $\partial_{\infty}\mathbb{H}^3$ under the bending map $\tilde{f}_{\beta}$.  Let $g$ be a geodesic in $\mathbb{H}^2$ whose endpoints are $x$ and $y$. 

If $g$ is a leaf of $\tilde{\lambda}$ then $\tilde{f}_{\beta}(x)\neq \tilde{f}_{\beta}(y)$ because the bending map is an isometry on leaves and plaques of $\tilde{\lambda}$.  

From now on we assume that $g$ transversely intersects $\tilde{\lambda}$. Then $g$ necessarily intersects plaques of $\tilde{\lambda}$ by the maximality(of $\tilde{\lambda}$). 

Consider a fixed geometric set of ties $\{ k_1,\ldots ,k_n\}$ for $\lambda$. If two ties are paired, connect their vertices by hyperbolic arcs to form a quadrilateral, called a {\it long rectangle}. The two paired ties are {\it short sides} and the other two sides are {\it long sides} of the long rectangle. The set of all long rectangles are edges and the set of ties are vertices of a {\it geometric train track} $\tau$. The geodesic lamination $\lambda$ is contained in the interior of the union of all edges (i.e. long rectangles) of $\tau$.  

Denote by $\tilde{\lambda}$ and $\tilde{\tau}$ the lifts to $\mathbb{H}^2$ of the geodesic lamination $\lambda$ and the geometric train track $\tau$. Then the geodesic lamination $\tilde{\lambda}$ is contained in the interior of the union of edges (i.e. long rectangles) of $\tilde{\tau}$.

If geodesic $g$ is completely contained in $\tilde{\tau}$ then $g$ is a leaf of 
$\tilde{\lambda}$. Therefore $g$ contains a point $p$ outside $\tilde{\tau}$. 
The point $p$ divides the geodesic $g$ into two rays $g_1$ and $g_2$ with 
the common initial point $p$ and endpoints $x$ and $y$, respectively. 

We normalize the bending map $\tilde{f}_{\beta}$ to be the identity on the 
plaque of $\tilde{\lambda}$ which contains $p$. Form two hyperbolic cones $
\mathcal{C}(p,g_1,\pi /2)$ and $\mathcal{C}(p,g_2,\pi /2)$ whose central 
axes are $g_1$ and $g_2$ with angle $\pi /2$.
A {\it shadow} of a cone is the set of endpoints (on the boundary $\partial_{\infty}\mathbb{H}^3$) of the rays contained in the cone. 
 We prove that $\tilde{f}
_{\beta}(x)$ stays in the (open) shadow at $\partial_{\infty}\mathbb{H}^3$ of the cone $
\mathcal{C}(p,g_1,\pi /2)$, and analogously for $\tilde{f}_{\beta}(y)$. This implies the desired result since the open shadows of $
\mathcal{C}(p,g_1,\pi /2)$ and $\mathcal{C}(p,g_2,\pi /2)$ are disjoint. 

We consider the ray $g_1$ and $\tilde{f}_{\beta}(x)$. The bending map 
$\tilde{f}_{\beta}$ is mapping the ray $g_1$ to a piecewise geodesic in $\mathbb{H}^3$ with the bending points $\tilde{\lambda}\cap g_1$ and the bending amount given by the finitely additive transverse cocycle $\beta$. The idea of controlling the position of $\tilde{f}_{\beta}(x)$ is to divide the ray $g_1$ into finite arcs such that the finitely additive ``$\beta$-measure'' on each arc is well-behaved with respect to the size of the arc. Consequently the bending map $\tilde{f}_{\beta}$ moves the endpoints of the arcs and the tangent vectors to $g_1$ at these endpoints such that 
$\tilde{f}_{\beta}(g_1)$ stays inside the cone $\mathcal{C}(p,g_1,\pi /2)$.

In more details,
consider the intersections of the ray $g_1$ with the boundary sides of the long rectangles (i.e edges of $\tilde{\tau}$) in the order from the initial point $p$. The first point of intersection $a_1$ of $g_1$ and (the union of long rectangles of) $\tilde{\tau}$ is on a long side of a long rectangle $E_1$. If $g_1$ exits $E_1$ through other long side of $E_1$, then denote by $b_1$ that point. If $g_1$ exits an adjacent edge $E_2$ through its long side then denote that point by $b_1$. If $g_1$ exits an edge $E_3$ adjacent to $E_2$ through its long side then denote the point of exit by $b_1$. Finally if $g_1$ exists $E_3$ through a short side, denote by $b_1$ the point of exit of $g_1$ from a short side of $E_2$.

We obtained the first arc $[a_1,b_1]$, denote by $E_{a_1}$ the edge of $\tilde{\tau}$ which $g_1$ enters at the point $a_1$ and denote by $E_{b_1}$ the edge which $g_1$ leaves at the point $b_1$. We take $a_2$ to be the point at which $g_1$ enters the first edge $E_{a_2}$ after the edge $E_{b_1}$. We note that it is possible that $b_1=a_2$ if $E_{b_1}$ and $E_{a_2}$ share a short edge. Then $b_2$ is determined analogously to $b_1$. 

We continue this process to obtain a sequence of disjoint arcs $\{ (a_n,b_n)\}$ on $g_1$ given in the increasing order for the orientation of $g_1$. If $b_n\neq a_{n+1}$ then the open arc $(b_n,a_{n+1})$ does not intersect $\tilde{\tau}$ and the closed arc $[b_n,a_{n+1}]$ does not intersect $\tilde{\lambda}$. According to our definition, either $[a_n,b_n]$ connects long sides of (possibly the same) long rectangle(s) and intersects at most three long rectangles, or $[a_n,b_n]$  connects two short sides of a long rectangle while intersecting at most three long rectangles (cf. \S 6).

Consider a sequence of nested cones $\{ \mathcal{C}(a_n,g_1,\pi /2)\supset \mathcal{C}(b_n,g_1,\pi /2)\}$. Let $P_{a_n}$ and $P_{b_n}$ be plaques containing $a_n$ and $b_n$, respectively. Then $\tilde{f}_{\beta}|_{P_{a_n}}$ and 
$\tilde{f}_{\beta}|_{P_{b_n}}$ are hyperbolic isometries of $\mathbb{H}^3$ such that 
$$
\tilde{f}_{\beta}|_{P_{b_n}}=\tilde{f}_{\beta}|_{P_{a_n}}\circ R_{[a_n,b_n]}
$$
where $R_{[a_n,b_n]}$ is a hyperbolic isometry defined using the transverse cocycle $\beta$ on the  arc $[a_n,b_n]$ (cf. Bonahon \cite{Bon2} and \S 4). 

We need to prove that
$$
\tilde{f}_{\beta}|_{P_{a_n}}(\mathcal{C}(a_n,g_1,\pi /2))\supset
\tilde{f}_{\beta}|_{P_{b_n}}(\mathcal{C}(b_n,g_1,\pi /2))
$$
for all $n$, which is equivalent to
\begin{equation}
\label{eq:main_inclusion}
\mathcal{C}(a_n,g_1,\pi /2)\supset R_{[a_n,b_n]}(\mathcal{C}(b_n,g_1,\pi /2))
\end{equation}
because the maps are hyperbolic isometries.

The core of the proof of the above theorem is bounding the isometry $R_{[a_n,b_n]}$ such that (\ref{eq:main_inclusion}) holds for all possible geodesics $g$ in $\mathbb{H}^2$ simultaneously. This is achieved by careful choice of a set of ties as in the theorem. To establish (\ref{eq:main_inclusion}), we consider different possibilities for the intersection of the arc $[a_n,b_n]$ with the edges of $\tilde{\tau}$. 

\vskip .2 cm

Assume that $[a_n,b_n]$ connects two long sides of a single long rectangle in $\tilde{\tau}$. There is a definite lower bound on the length of $[a_n,b_n]$ in terms of $w_{*}$ and $l^{*}$ (cf. Lemma \ref{lem:size_rectangles}).
Lemma \ref{lem:nested_cones_for_bounded_distances} gives a numerical bound on the distance of $R_{[a_n,b_n]}$ from the identity (which is a constant multiple of the length of $[a_n,b_n]$)  such that (\ref{eq:main_inclusion}) holds. 

Denote by $T^c_g$ the hyperbolic isometry with the axis $g$ and the translation length $c\in\mathbb{C}$. 
The isometry $R_{[a_n,b_n]}$ is given by (cf. Bonahon \cite{Bon2} and \S 4).
$$
R_{[a_n,b_n]}=\lim_{i\to\infty}B_1B_2\cdots B_iR_{P_{b_n}}
$$ 
where 
$$
B_j=R_{g_{d_j}^{P_{a_n}}}^{\beta (P_{a_n},P_{d_j})} R_{g_{d_j}^{P_{b_n}}}^{-\beta (P_{a_n},P_{d_j})}
$$
with $P_{d_j}$ being the plaque containing the gap $d_j$, $\beta (P_{a_n},P_{d_j})$ being the $\beta$-mass of a closed arc with endpoints in $P_{a_n}$ and $P_{d_j}$
for $j=1,\ldots ,i$; and where
$$
R_{P_{b_n}}=R^{\beta ([a_n,b_n])}_{g_{P_{b_n}}^{P_{a_n}}}.
$$

Lemma \ref{lem:const_in_norm} and \ref{lem:bound_on_norm} give estimates for $B_j$ and $T_{P_{b_n}}$ which allows us to conclude that $\|R_{[a_n,b_n]}-Id\|$ is bounded by a linear function with variables $|\beta ([a_n,b_n])|\leq \|\beta\|_{max}$, $\|\beta\|_{var_{\delta}}$ and $\delta >0$ (cf. \S 6). This allows us to choose universal $\epsilon >0$ and $\delta >0$ such that the assumptions of Lemma \ref{lem:nested_cones_for_bounded_distances} are satisfied and we obtain 
(\ref{eq:main_inclusion}) in this case.

\vskip .2 cm

Assume that $[a_n,b_n]$ enters the edge $E_1$ at a short side, then enters the adjacent edge $E_2$ at a short side, and it exits $E_2$ at a long edge. The composition
$$
B_1B_2\cdots B_i$$
is estimated in the same fashion as above. However, the hyperbolic rotation
$$
R_{{P_{b_n}}}
$$
can have arbitrary large angle $\beta ([a_n,b_n])$ since $\beta$ is finitely additive transverse cocycle and $[a_n,b_n]$ does not cross the set of  all geodesics following a single edge of $\tilde{\tau}$. In order to have a control on where $\mathcal{C}(b_n,g_1,\pi /2)$ is mapped by $R_{{P_{b_n}}}$, it is necessary that the angles of intersections between the geodesics of $\tilde{\lambda}$ and $[a_n,b_n]$ are small (cf. Lemma \ref{lem:rotation_large_angle}). The angles are made small enough by the choice of $w^{*}$ and $l^{*}$ in the main theorem. Thus we have the inclusion of the cones again. All other cases are dealt by combining the above two case with slightly larger constants which gives the proof of the main theorem (cf. \S 6).

\vskip .2 cm

\paragraph{\it Acknowledgements.} I am grateful to F. Bonahon and V. Markovic for various discussions regarding the previous version of this article. I am also grateful to the referee for his comments.

\section{Geodesic laminations}

Given a hyperbolic surface and a (maximal) geodesic lamination on the surface, 
we define a ``metric'' train track using rectangles as edges and sides at which two rectangles meet (called short sides of rectangle) as vertices. The geodesic lamination will be contained in the interior of the union of rectangles and the angles at which geodesics of the lamination meet the vertices (i.e. short sides of rectangles) are bounded away from $0$ and $\pi$. In the lemma below, we find a lower bound on the distance between two opposite sides of a rectangle(the sides that are not short) in terms of the diameters of the rectangle and  the lengths of the short sides.

\vskip .2 cm

Let $S$ be a closed hyperbolic surface and $\lambda$ a maximal geodesic
lamination on $S$. Each component of
$S\setminus\lambda$, called {\it plaque} of $\lambda$, is an ideal hyperbolic triangle
for the path metric of the complement.
Let $\{ h_1,\ldots ,h_m\}$ be a set of finite geodesic arcs on $S$ with endpoints in the plaques of $\lambda$ such that each geodesic of $\lambda$ is divided into finite length arcs by the set $\cup_{i=1}^mh_i$. The family of arcs $\lambda\setminus\cup_{i=1}^mh_i$ consists of finitely many homotopy classes relative $\{ h_1,\ldots ,h_m\}$ and we assume that after identifying all the arcs of the homotopy classes the obtained ``topological train track'' has the property that each vertex (corresponding to some $h_i$) is either  trivalent or bivalent. The usual definition of train tracks does not allow bivalent vertices but we do allow them. The reason is that we need more flexibility to obtain a train track which has good geometric properties. 

We form a ``metric train track'' $\tau$ as follows. A {\it gap} of $h_i$ (with respect to $\lambda$) is a connected component of $h_i\setminus\lambda$. If $h_i$ corresponds to a trivalent vertex of the corresponding topological train track, we divide arc $h_i$ into two subarcs $h_i^1$ and $h_i^2$ with a division point in a gap of $h_i$ such that the endpoints of the arcs of $\lambda\setminus\cup_{i=1}^mh_i$ which belong to different homotopy classes lie in different subarcs. For the convenience of notation, denote by $\{ k_1,\ldots ,k_n\}$ the set of all arcs $h_i,h_i^j$ for $i=1,\ldots ,m$ and $j=1,2$.

We connect the endpoints of $\{ k_1,\ldots ,k_n\}$ by geodesic arcs inside the plaques of $\lambda$ (whenever this is possible)  to obtain a finite collection of geodesic quadrilaterals whose two sides are among $k_i$'s and the other two sides are obtained by connecting the chosen points on $k_i$'s inside the components of $S\setminus (\lambda\cup\bigcup_{i=1}^nk_i)$. We call these quadrilaterals, somewhat improperly, {\it long rectangles}. The sides of the rectangles which are among $k_i$'s are said to be {\it short} and the other two sides are said to be {\it long}. The finite collection of long rectangles forms a {\it (metric) train track} $\tau$ on $S$ such that the long rectangles are the {\it edges} of $\tau$ and the {\it switches} of $\tau$ are the arcs $\{ k_1,\ldots ,k_n\}$, where we allow switches to be either trivalent or bivalent. The geodesic lamination $\lambda$ is a subset of the interior of (the union of the edges of) the train track $\tau$ and it is said that $\lambda$ is {\it carried} by $\tau$. The train tack $\tau$ is homotopic to the topological train track. This kind of train tracks were introduced by Bonahon (cf. \cite{Bon1}).

\begin{definition}
Let $l^{*}$ be the maximum of the diameters of the long rectangles of $\tau$ and let $l_{*}$ be the shortest distance between two short sides of the long rectangles. Let 
$w_{*}$ be the minimum of the lengths of the short sides over all long rectangles of $\tau$ and let $w^{*}$ be the maximum of the lengths of the short sides over all long rectangles of $\tau$.
\end{definition}

We impose two conditions on $\tau$. Namely we require that
the angles at the vertices of each long rectangle lie in the interval $[\pi /4,3\pi /4]$
and that $$w^{*}\leq\frac{1}{20}.$$ A (metric) train track which satisfies these conditions is said to be {\it geometric} and the corresponding collection of arcs $\{ k_1,k_2,\ldots ,k_n\}$ is said to be {\it geometric}. We note that the angles are given by the choice of the arcs $\{ h_1,\ldots ,h_m\}$ above and that $w^{*}$ can be made small enough by further subdividing the arcs $\{ h_1,\ldots ,h_m\}$, if necessary.

\begin{lemma}
\label{lem:size_rectangles}
Let $R$ be a long rectangle of a geometric train track $\tau$ with short sides $k_1$ and $k_2$. Then the distance $d$ between the long sides of $R$ satisfies
$$
d\geq\frac{1}{20e^{l^{*}}}
\min\{ |k_1|,|k_2|\},
$$
where $|k_i|$ is the length of $k_i$.
\end{lemma}

\begin{proof}
Let $k_1$ and $k_2$ be the short sides of $R$, and let $l_1$ and $l_2$ be the long sides of $R$. Denote by $h_1$ and $h_2$ orthogonal arcs from $k_1\cap l_2$ and $k_2\cap l_2$ onto $l_1$, respectively. The hyperbolic sine formula, the bounds on the angles at the vertices of $R$ and the mean value theorem give
$$
|h_i|\cosh 1\geq\sinh |h_i|\geq \frac{1}{\sqrt{2}}\sinh |k_i|\geq  \frac{1}{\sqrt{2}}|k_{i}|
$$
which gives
$$
|h_i|\geq\frac{1}{\sqrt{2}\cosh 1}|k_{i}|.
$$
By possibly decreasing $R$, we can assume that $h_1$ and $h_2$ have the same length $|h_1|=|h_2|\geq \frac{1}{\sqrt{2}\cosh 1}\min\{|k_{1}|,|k_2|\}$. Let $h$ be the arc which is orthogonal to both $l_1$ and $l_2$. An elementary hyperbolic geometry formula applied to the rectangle whose two sides are $h_1$ and $h$ gives
$$
|h|\cosh 1\geq\sinh |h|\geq\frac{\sinh\frac{\min\{|k_{1}|,|k_2|\}}{\sqrt{2}\cosh 1}}{\sqrt{\sinh^2\frac{l^{*}+4\min\{|k_{1}|,|k_2|\}}{2}\cosh^2( \min\{|k_{1}|,|k_2|\})+1}}
$$
which in turn gives
$$
d\geq |h|\geq\frac{1}{2(\cosh^2 1)e^{l^{*}+1}}\min\{|k_{1}|,|k_2|\}.
$$
\end{proof}

\section{Transverse cocycles to geodesic laminations}

In this section we define finitely additive transverse cocycles to a geodesic lamination 
$\lambda$ on the surface $S$ (cf. Bonahon \cite{Bon1}). The transverse cocycles that we consider are $\mathbb{R}$-valued, $(\mathbb{R}/2\pi \mathbb{Z})$-valued and $(\mathbb{C}/2\pi i\mathbb{Z})$-valued.

\begin{definition} \cite{Bon2}
A {\it real valued transverse cocycle} $\alpha$ to $\lambda$ is an
assignment of a finitely additive real valued measure to each 
arc $k$ transverse to $\lambda$ which is invariant under homotopies relative
$\lambda$. Namely, $\alpha$ assigns a real number $\alpha (k)$ to
each closed arc $k$ transverse to $\lambda$ whose endpoints are in
$S-\lambda$ such that if $k'$ is an arc homotopic to $k$ relative
$\lambda$ then $\alpha (k)=\alpha (k')$. Moreover if $k=k_1\cup
k_2$, where $k_1$ and $k_2$ are transverse arcs to $\lambda$ with
disjoint interiors, then $\alpha (k)=\alpha (k_1)+\alpha (k_2)$.
\end{definition}

\begin{remark}
Let $\alpha$ be a finitely additive transverse cocycle  to a geodesic lamination $\lambda$. Let $k$ be a closed arc  transverse to $\lambda$ such that $|\alpha (k)|\neq 0$ and $\alpha$ is not countably additive on $k$. Then there exists a sequence of subarcs $\{ k_n\}$ of $k$ such that 
$$|k_n|\to 0$$  and 
$$|\alpha (k_n)|\to\infty$$ as $n\to\infty$. This phenomenon does not appear for the countably additive transverse cocycles and it forces delicate arguments when working with finitely additive cocycles (cf. \cite{Bon1}, \cite{Bon2}, \cite{Bon3}).
\end{remark}

A hyperbolic metric on $S$ induces a real valued transverse cocycle to a maximal geodesic lamination $\lambda$ as follows
(cf. \cite{Bon2}). 

Each complementary triangle (i.e. plaque)
of $\lambda$ can be partially foliated by three families of horocyclic arcs with centers 
at the vertices of the complementary triangle such that the portion which is not foliated is a finite triangle containing the center of the complementary ideal hyperbolic triangle. 

Every point  of every boundary geodesic of a plaque lies in exactly one
horocyclic leaf of the partial foliation except for the vertices of the finite triangle where two  horocyclic leaves meet (cf. Bonahon \cite{Bon1}). The partial foliation of the plaques extends to the leaves of $\lambda$ by the continuity and the surface is foliated except for the finite triangles (with horocyclic boundaries) inside the plaques. Note that the boundary geodesics of the plaques are leaves of $\lambda$. However, in general, geodesic laminations can have uncountably many leaves and the extension of the foliation is non trivial (cf. \cite{Bon2}).

Lift a maximal geodesic lamination $\lambda$ to maximal geodesic lamination $\tilde{\lambda}$ of the hyperbolic plane $\mathbb{H}^2$.
Given a closed hyperbolic arc $k$ on $S$ with endpoints in plaques of $\lambda$ which transversely intersects $\lambda$, denote by $\tilde{k}$ its lift to $\mathbb{H}^2$. Let $P_1$ and $P_2$ be the plaques of $\tilde{\lambda}$ which contain the endpoints of $\tilde{k}$. Let $p_i$ be the vertex of the central triangle of $P_i$ on the boundary of $P_i$ facing $P_{i+1}$, for $i=1,2$ (where $i+1$ is taken modulo $2$). Let $p_1'$ be the point where the leaf of the (lifted) horocyclic foliation through the point $p_1$ meets the boundary of $P_2$ facing $P_1$. The value of the transverse cocycle to the arc $k$ is the signed distance between $p_2$ and $p_1'$ when the geodesic through them is oriented as a part of the boundary of $P_2$ (cf. \cite{Bon2}).

A finitely additive real valued transverse cocycle to a maximal geodesic lamination $\lambda$ induces a transverse H\"older distribution to $\lambda$ (i.e. a linear functional on H\"older continuous functions on each transverse arc to $\lambda$ that  is invariant under homotopies relative $\lambda$). Conversely, a transverse H\"older distribution to $\lambda$ induces a real valued finitely additive transverse cocycle to $\lambda$ (cf. \cite{Bon3}). 

Denote by $\mathcal{H}(\lambda ,\mathbb{R})$ the space of real valued finitely additive transverse cocycles to $\lambda$. The subset of transverse cocycles which arise from hyperbolic metrics on $S$ is an open cone in $\mathcal{H}(\lambda ,\mathbb{R})$ denoted by $\mathcal{C}(\lambda )$ (cf. \cite{Bon2}). In fact, $\mathcal{C}(\lambda )$ parametrizes the Teichm\"uller space $T(S)$ of the closed surface $S$ (cf. \cite{Bon2}). Given a point in $\mathcal{C}(\lambda )$, there is a procedure of recovering the corresponding metric on $S$ by constructing the corresponding 
representation of the fundamental group $\pi_1(S)$ using the transverse cocycle (cf. \cite{Bon2}). 

\vskip .2 cm

We will also need finitely additive transverse cocycles which take values in
$\mathbb{R}/2\pi\mathbb{Z}$ and in $\mathbb{C}/2\pi i\mathbb{Z}$.

\begin{definition} \cite{Bon2}
An $(\mathbb{R}/2\pi\mathbb{Z})$-{\it valued transverse cocycle}
$\beta$ for $\lambda$ is an assignment of $\beta (k)\in
\mathbb{R}/2\pi\mathbb{Z}$ to each transverse arc $k$ to $\lambda$
which is invariant under homotopy relative $\lambda$ and which is
finitely additive. For example, an
$(\mathbb{R}/2\pi\mathbb{Z})$-valued transverse cocycle $\beta$ is
obtained by taking a real valued transverse cocycle $\alpha$ and
setting $\beta (k):=\alpha (k)\mod (2\pi\mathbb{Z})$.

Similarly, a $(\mathbb{C}/2\pi i\mathbb{Z})$-{\it valued transverse cocycle}
$\beta$ for $\lambda$ is an assignment of $\beta (k)\in
\mathbb{R}/2\pi\mathbb{Z}$ to each transverse arc $k$ to $\lambda$
which is invariant under homotopy relative $\lambda$ and which is
finitely additive.
\end{definition}

A {\it pleated surface} with the {\it pleating locus} $\lambda$ is a continuous map
$$
\tilde{f}:\mathbb{H}^2\to\mathbb{H}^3
$$
which is an isometry on each leaf and on each plaque of the lift $\tilde{\lambda}=\pi^{-1}(\lambda )$ for the universal covering $\pi :\mathbb{H}^2\to S$, and which conjugates the covering group $G$ of $S$ into a subgroup $fGf^{-1}$ of $PSL_2(\mathbb{C})$. 

Each plaque of $\tilde{\lambda}$ is isometrically mapped by $\tilde{f}$ to an ideal hyperbolic triangle in $\mathbb{H}^3$ and the amount of bending along $\tilde{\lambda}$ is measured by an $(\mathbb{R}/2\pi \mathbb{Z})$-valued finitely additive transverse cocycle to $\tilde{\lambda}$. The pleating map can be recovered from an $(\mathbb{R}/2\pi \mathbb{Z})$-valued finitely additive transverse cocycle by constructing the isometries of $\mathbb{H}^3$ which are equal to the restriction of the pleating map on the given plaque (cf. Bonahon \cite{Bon1} and \S 4). Moreover, a pleating map with the pleating locus $\lambda$ induces a $(\mathbb{C}/2\pi i\mathbb{Z})$-valued finitely additive transverse cocycle to $\tilde{\lambda}$ with the real part determining the path metric on $\tilde{f}(\mathbb{H}^2)$(when considered as a subset of the hyperbolic three space $\mathbb{H}^3$) and the imaginary part giving the amount of bending along $\tilde{\lambda}$.

\vskip .2 cm

Let $\alpha$ be either an $\mathbb{R}$-valued or an
$(\mathbb{R}/2\pi\mathbb{Z})$-valued transverse cocycle for $\lambda$ and let $\tau$ be a geometric train track that carries $\lambda$.
Given an edge $E\in\tau$, let $k_E$ be a geodesic arc which connects the two long boundary sides of $E$. Define
$\alpha (E):=\alpha (k_E)$. Note that $\alpha (E)$ is independent of
the choice of $k_E$ by the invariance of $\alpha$ under homotopy relative $\lambda$. The transverse cocycle $\alpha$ is
completely determined by the values $\alpha (E)$, $E\in\tau$(cf.
\cite{Bon1}).

\section{The realizations of $\mathbb{R}$-valued and $(\mathbb{R}/2\pi\mathbb{Z})$-valued transverse cocycles}

The purpose of this section is to recall the procedure of constructing the realizations of $\mathbb{R}$-valued and  $(\mathbb{R}/2\pi\mathbb{Z})$-valued transverse cocycle to a maximal geodesic lamination on a closed surface $S$ (cf. \cite{Bon2}). In the former case we obtain a representation of $\pi_1(S)$ into $PSL_2(\mathbb{R})$ which gives a hyperbolic metric on $S$. In the later case, we start from a hyperbolic metric on $S$ and obtain a pleated surface with the pleating locus $\lambda$ and the bending amount given by the  $(\mathbb{R}/2\pi\mathbb{Z})$-valued transverse cocycle. The core argument in the proof of the main theorem is estimating the pleating map arising from an  $(\mathbb{R}/2\pi\mathbb{Z})$-valued transverse cocycle which satisfies certain geometric conditions.

We consider a hyperbolic surface $S$ and a maximal geodesic
lamination $\lambda$ of $S$. Bonahon \cite{Bon2} defined an
injective map from the Teichm\"uller space $T(S)$ of $S$ into the
space $\mathcal{H}(\lambda ,\mathbb{R})$ of all $\mathbb{R}$-valued
transverse cocycles to a fixed maximal geodesic lamination $\lambda$ which is a
homeomorphism onto an open cone $\mathcal{C}(\lambda )$ of
$\mathcal{H}(\lambda ,\mathbb{R})$. 

Denote by $\sigma_0$ a fixed
hyperbolic metric on $S$. Then $\sigma_0$ represents the base point in
$T(S)$ and let $\alpha_0\in\mathcal{H}(\lambda ,\mathbb{R})$ be the
corresponding transverse cocycle. Then \cite[Proposition
13]{Bon2} any other real-valued cocycle $\alpha\in
\mathcal{H}(\lambda ,\mathbb{R})$ which is close enough to
$\alpha_0$ is also in the image of $T(S)$ in $\mathcal{H}(\lambda
,\mathbb{R})$. Namely, when the difference $\alpha -\alpha_0$ is
 small in the sense that the norm $\|\alpha -\alpha_0\|_{\max}$ is
small, where
$$\|\alpha\|_{\max}
:=\max_{E}|\alpha (E)|$$ and the maximum is over all edges $E$ of a train track
$\tau$ that carries $\lambda$ then $\alpha$ determines a point in $T(S)$.

The proof of the above statement is given by constructing the
realization of $\alpha_1$ starting from the realization of $\alpha_0$.
We recall that $\lambda$ is a maximal geodesic lamination for the
metric $\sigma_0$. We lift $\lambda$ to a geodesic lamination
$\tilde{\lambda}$ of the universal covering $\mathbb{H}^2$.
Components of $\mathbb{H}^2\setminus\tilde{\lambda}$ are called {\it
plaques} of $\tilde{\lambda}$ and they are lifts of plaques (i.e. connected components of $S\setminus\lambda$) of
$\lambda$. Each plaque of $\tilde{\lambda}$ is an ideal hyperbolic
triangle. 

Let $k$ be an oriented geodesic arc in $\mathbb{H}^2$ from
plaque $P$ to plaque $Q$ of $\tilde{\lambda}$. Denote by
$\mathcal{P}_{P,Q}$ be the set of all plaques of $\tilde{\lambda}$
that separate $P$ and $Q$, and by $\mathcal{P}=\{P_1,P_2,\ldots
,P_n\}$ a finite subset of $\mathcal{P}_{P,Q}$ such that the
indices increase from $P$ to $Q$. Let $g_i^P$ and $g_i^Q$ be the
geodesics on the boundary of the plaque $P_i$ that separate $P_i$
from $P$ and $Q$, respectively. Let $g_Q^P$ be the geodesic on the
boundary of $Q$ that separates $P$ and $Q$. Define $\alpha=\alpha_1-\alpha_0$. Let $\alpha (P,P_i)$
denote the $\alpha$-mass of a geodesic arc with endpoints in $P$ and
$P_i$; similar definition for $\alpha (P,Q)$. 

Let
$$
D_{P_i}=T_{g_i^P}^{\alpha (P,P_i)}T_{g_i^Q}^{-\alpha
(P,P_i)}
$$
where $T_{g_i^P}^{\alpha (P,P_i)}$, $T_{g_i^Q}^{-\alpha (P,P_i)}$ are hyperbolic translations(in $\mathbb{H}^2$)
with axes $g_i^P$, $g_i^Q$ which are oriented to the left as seen from $P$, $Q$ and with
the translation lengths $\alpha (P,P_i)$, $-\alpha (P,P_i)$.
Let
$$
T_Q=T_{g_Q^P}^{\alpha (P,Q)}.
$$

A finite approximation $\varphi_{\mathcal{P}}$ to the realization of the transverse cocycle $\alpha_1$ corresponding to a finite set of plaques $\mathcal{P}$ is given by
$$
\varphi_{\mathcal{P}}=D_{P_1}D_{P_2}\cdots
D_{P_n}T_Q.
$$

The realization $\varphi_{P,Q}$ of the transverse cocycle $\alpha$ is given by (cf.
\cite{Bon2})
$$\varphi_{P,Q}=\lim_{\mathcal{P}\to\mathcal{P}_{P,Q}}\varphi_{\mathcal{P}}.$$

Let
$$
\psi_{\mathcal{P}}=D_{P_1}D_{P_2}\cdots
D_{P_n}
$$
and let
$$\psi_{P,Q}=\lim_{\mathcal{P}\to\mathcal{P}_{P,Q}}\psi_{\mathcal{P}}.$$

It follows that
$$
\varphi_{P,Q}=\psi_{P,Q}\circ T_{g_Q^P}^{\alpha (P,Q)}.
$$

The quantity $\psi_{P,Q}$ is the {\it difference from
$T_Q$ of the realization $\varphi_{P,Q}$} of hyperbolic metric
$\sigma$ on $S$. Bonahon \cite{Bon2} proved that for a fixed surface $S$
and small {\it norm}
$$\|\alpha_1 -\alpha_0\|_{\max}=\max_{E\in\tau}|\alpha_1 (E)-\alpha_0(E)|$$ the difference $\psi_{P,Q}$ always lies in a compact
subset of $PSL_2(\mathbb{R})$. 

\vskip .2 cm

Our main interest are bending pleated surfaces. The bending of (abstract) pleated surfaces with the
pleating locus $\lambda$ is completely determined by an
$(\mathbb{R}/2\pi\mathbb{Z})$-valued transverse cocycles for the
geodesic lamination $\lambda$ on $S$ and each
$(\mathbb{R}/2\pi\mathbb{Z})$-valued transverse cocycle to $\lambda$
is realized by an abstract pleated surface with pleating locus
$\lambda$ (cf. \cite{Bon2}). Denote by $\mathcal{H}(\lambda
,\mathbb{R}/2\pi\mathbb{Z})$ the space of all $(\mathbb{R}/2\pi\mathbb{Z})$-valued finitely additive transverse cocycles to the lamination $\lambda$. The space of all
abstract pleated surfaces with the pleating locus $\lambda$ is
parametrized by $\mathcal{C}(\lambda )\oplus i\mathcal{H}(\lambda
,\mathbb{R}/2\pi\mathbb{Z})$, where $\mathcal{C}(\lambda )$ is the
open cone in $\mathcal{H}(\lambda ,\mathbb{R})$ which parametrizes
the Teichm\"uller space $T(S)$ (cf. \cite{Bon2}).

Let $\beta\in\mathcal{H}(\lambda ,\mathbb{R}/2\pi\mathbb{Z})$ and denote by $R_g^a$ hyperbolic rotation in $\mathbb{H}^3$ with the axis $g$ and the rotation angle $a\in\mathbb{R}$. Let $\mathcal{P}_{P,Q}$ be the set of all plaques of $\tilde{\lambda}$ separating plaques $P$ and $Q$, and let $\mathcal{P} =\{ P_1,P_2,\ldots ,P_n\}$ be a finite subset of $\mathcal{P}_{P,Q}.$

Define
$$
B_{P_i}=R_{g_i^P}^{\beta (P,P_i)}R_{g_i^Q}^{-\beta
(P,P_i)}
$$
where $P_i\in\mathcal{P}$, $\beta (P,P_i)$ the $\beta$-mass of a geodesic arc connecting $P$ and $P_i$, $g_i^P$ the geodesic on the boundary of $P_i$ facing $P$, and $g_i^Q$ the geodesic on the boundary of $P_i$ facing $Q$. 
Define
$$R_Q=R_{g_Q^P}^{\beta
(P,Q)}.
$$

A finite approximation $\varphi_{\mathcal{P}}$  of the
realization of the bending cocycle $\beta$ is defined by 
$$
\varphi_{\mathcal{P}}=B_{P_1}B_{P_2}\cdots B_{P_n}R_Q
$$
and the realization $\mathcal{P}_{P,Q}$ of $\beta$ is defined by (cf. \cite{Bon2})
$$\varphi_{P,Q}=\lim_{\mathcal{P}\to\mathcal{P}_{P,Q}}\varphi_{\mathcal{P}}$$

Note that the
realization exists for all $\beta\in\mathcal{H}(\lambda
,\mathbb{R}/2\pi\mathbb{Z})$ because of the compactness of
$\mathbb{R}/2\pi\mathbb{R}$. The {\it difference from
$R_Q$ of the realization} of $\beta$ is given by
\begin{equation}
\label{eq:diff_for_realization}
\psi_{P,Q}=\lim_{\mathcal{P}\to\mathcal{P}_{P,Q}}\psi_{\mathcal{P}}
\end{equation}
where
\begin{equation}
\label{eq:diff_for_realization_1}
\psi_{\mathcal{P}}=B_{P_1}B_{P_2}\cdots B_{P_n}.
\end{equation}

The bending map $\tilde{f}:\mathbb{H}^2\to\mathbb{H}^3$ is defined by fixing a plaque $P$ and setting $\tilde{f}|_Q=\varphi_{P,Q}$ for any plaque $Q$. Note that $\tilde{f}|_P=id$.

Let $\tilde{f}:\mathbb{H}^2\to\mathbb{H}^3$ be the bending map for
the bending pleated surface defined by $\beta$, where $\mathbb{H}^2$
is identified with the $(xz)$-half-plane in the upper half-space
$\mathbb{H}^3=\{ (z,t):z\in\mathbb{C},t>0\}$. Then $\tilde{f}$ does
not necessarily extend to an injective map from
$\partial_{\infty}\mathbb{H}^2$ into $\partial_{\infty}\mathbb{H}^3$. The core argument in the proof of the main theorem establishes that the geometric
 conditions on $\beta$ guarantee
injectivity of the continuous extension $\tilde{f}:\partial_{\infty}\mathbb{H}^2\to\partial_{\infty}\mathbb{H}^3$ (which is achieved in \S 6).

\section{The nested cones}

In this section we prove several lemmas needed in the proof of the main theorem in \S 6. As discussed in Introduction, the main argument establishes the injectivity of the bending map $\tilde{f}:\mathbb{H}^2\to\mathbb{H}^3$. Namely, given $x,y\in\partial_{\infty}\mathbb{H}^2$ with $x\neq y$, we need to prove that $\tilde{f}(x)\neq\tilde{f}(y)$. We consider the geodesic $g$ in $\mathbb{H}^2$ whose ideal endpoints are $x$ and $y$, and prove that the image of $g$ under the bending map $\tilde{f}$ behaves well enough to have distinct endpoints in $\partial_{\infty}\mathbb{H}^3$ (cf. \S 6). 

Let $p\in g$ be a point in a plaque of $\tilde{\lambda}$, and let $g_1,g_2$ be the two geodesic rays that $g$ is divided into by the point $p$.
We consider two hyperbolic cones both having vertex $p$, angle $\pi /2$ and axes $g_1$ and $g_2$. The open shadows of the two cones are disjoint (in fact, each open shadow is a component of the complement of the boundary circle of an embedded hyperbolic plane in $\mathbb{H}^3$ containing point $p$ and orthogonal to $g$), and $x$ belongs to one and $y$ belongs to the other shadow. In order to prove that $\tilde{f}(x)\neq\tilde{f}(y)$, we normalize $\tilde{f}$ to be the identity on the plaque containing $p$ and prove that the image $\tilde{f}(g_i)$, for $i=1,2$, stays in the corresponding hyperbolic cone.

To do so, we divide $g_i$ into disjoint subarcs and consider a sequence of nested cones at the endpoints of these subarcs with the same angle $\pi /2$. 
The goal is to prove that the image under $\tilde{f}$ of the nested sequence of cones remains nested. It is enough to consider the consecutive hyperbolic cones and prove they stay embedded.

There are essentially two different phenomenon that can occur for the transverse cocycle $\beta$ along the sequence of arcs on $g_i$. Each arc either intersects an edge $E$ of the geometric train track such that it connects the long sides of the long rectangle $E$(in which case the arc intersects exactly the set of geodesics of $\tilde{\lambda}$ traversing $E$) or it connects two short sides of a long rectangle $E$ while intersecting only a portion of geodesic that traverse $E$.

In Lemma \ref{lem:nested_cones_for_bounded_distances}, we prove that if an isometry $A$ of $\mathbb{H}^3$ is close to the identity then the image of the inside cone stays in the outside cone as long as the distance between the vertices of the two cones is comparable to the size of $\| A-id\|$. Lemma \ref{lem:small_mobius_action} is used in the proof of Lemma \ref{lem:nested_cones_for_bounded_distances}. 

In the former case, the realization of the cocycle $\beta$ is on the distance from the identity comparable to the quantities $\|\beta\|_{\max}$ and $\|\beta\|_{var_{\delta}}$ by Lemma 5.4 and 5.5. Since the arc connects two long sides of a long rectangle, Lemma 2.1 applies to conclude that there is a lower bound on the length of the arc. Then Lemma \ref{lem:nested_cones_for_bounded_distances} implies the desired nesting of the cones. 

In the later case, the $\beta$-mass of the arc can be arbitrary large. In order to estimate the realization $\varphi_{P,Q}$ where $P$ and $Q$ are plaques containing the endpoints of the arc, we recall $\varphi_{P,Q}=\psi_{P,Q}R_Q$. The isometry $\psi_{P,Q}$ is approximated by $\psi_{\mathcal{P}}=B_1B_2\cdots B_n$ and the above argument proves that $\psi_{P,Q}$ is close enough to the identity when $\|\beta\|_{max}$ and $\|\beta\|_{var_{\delta}}$ are small enough. The rotation $R_Q$ does not have small angle since the $\beta$-mass can be large on the arc and Lemma 5.4 does not apply. Instead, since the distance between the short sides of each long rectangle is long, Lemma 5.3 gives the nesting of the cones in this case. We note that we do not have this phenomenon when the transverse cocycle is countably additive (i.e. it is a measure). 

\vskip .2 cm

Let $g\subset\mathbb{H}^3$ be a geodesic ray with initial point
$p_0$, and let $p$ be a point on the ray $g$. For $0<\theta <\pi$, the {\it
cone $\mathcal{C}(p,g,\theta )$ with vertex $p$, axis $g$ and angle
$\theta$} is the set of all $w\in\mathbb{H}^3$ such that the angle
at $p$ between the positive direction of $g$ and the geodesic ray
from $p$ through $w$ is less than $\theta$. A non-zero vector
$(p,v)\in T^1(\mathbb{H}^3)$ uniquely determines a geodesic ray
$g$ with the basepoint $p$ tangent to
$v$. By definition $\mathcal{C}(p,v,\theta )$ is 
$\mathcal{C}(p,g,\theta )$. The {\it shadow of the cone}
$\mathcal{C}(p,g,\theta )$ is the set
$\partial_{\infty}\mathcal{C}(p,g,\theta )$ of endpoints at
$\partial_{\infty}\mathbb{H}^3$ of all geodesic rays starting at $p$
and inside $\mathcal{C}(p,g,\theta )$.

For $d>0$, let $p_d\in g$ be the point on $g$ which is on the
distance $d$ from $p_0=p$. Let $\eta >0$ be the maximal angle such
that $\mathcal{C}(p_d,g,\eta )\subset\mathcal{C}(p_0,g,\theta )$.
Then $\eta =\eta (d,\theta )$ is a continuous function of $d$ and
$\theta$. For a fixed $0<\theta <2\pi$, we have $\eta (d,\theta
)>\theta$ and $\eta (d,\theta )\to\theta$ as $d\to 0$ (cf. \cite{KeeSe}).

We use quaternions to represent the upper half-space model
$\mathbb{H}^3=\{ z+tj:z\in\mathbb{C},t>0\}$ of the hyperbolic
three-space $\mathbb{H}^3$ (cf. Beardon \cite{Bear}). The space of isometries of
$\mathbb{H}^3$ is identified with $PSL_2(\mathbb{C})$ which is
equipped with the norm
$$
\| A\| =\max\{ |a|+|b|,|c|+|d|\}
$$
where $A(z)=\frac{az+b}{cz+d}\in PSL_2(\mathbb{C})$ and $ad-bc=1$.
The Poincar\'e extension to $\mathbb{H}^3$ of the action of $A\in
PSL_2(\mathbb{C})$ on $\hat{\mathbb{C}}$ is computed in \cite{Bear}
to be
$$A(z+tj)=\frac{(az+b)\overline{(cz+d)}+a\bar{c}t^2+tj}{|cz+d|^2+|c|^2t^2}.
$$

An isometry of $\mathbb{H}^3$ which is close to the identity moves
points on a bounded distance from $j\in\mathbb{H}^3$ by a small
amount and the tangent vectors are rotated by a small angle with
respect to the Euclidean parallel transport in $\mathbb{R}^3$. In Lemma \ref{lem:small_mobius_action}, we
give a quantitative statement of the above fact including the situation
when the points are not on the bounded distances from
$j\in\mathbb{H}^3$.

Given $p=z+tj\in\mathbb{H}^3$, we define
$$
ht(p)=t
$$
and
$$
Z(p)=z.
$$

Denote by $T_p\mathbb{H}^3$ the tangent space at the point $p\in\mathbb{H}^3$. The unit tangent space $T^1_p\mathbb{H}^3$ at the point $p$ is the quotient $(T_p\mathbb{H}^3-\{ 0\})/\mathbb{R}^{+}$ of non-zero tangent vectors $T_p\mathbb{H}^3-\{ 0\}$ by the positive real numbers $\mathbb{R}^{+}$. If $(p,u)\in T_p\mathbb{H}^3$, then $(p,u)/\sim =\{ (p,tu)|t\in\mathbb{R}^{+}\}$ is the corresponding unit tangent vector. For simplicity of notation, we denote by $(p,u)$ the unit tangent vector corresponding to $(p,u)$. Denote by $T^1\mathbb{H}^3$ the unit tangent space to $\mathbb{H}^3$.

Given $(p,v),(q,w)\in T^1\mathbb{H}^3$, define the distance on $T\mathbb{H}^3$ by
$$
D_{T^1\mathbb{H}^3}((p,u),(q,v))=\max\{ |ht(p)-ht(q)|,|Z(p)-Z(q)|,|\angle (u,v)|\}$$
where $\angle (u,v)$ is the angle between the vectors $u$ and the euclidean parallel translate of $v$ at the point $p$.

If an isometry of $\mathbb{H}^3$ is close to the identity, then the unit tangent vectors in a compact subset of $T^1\mathbb{H}^3$ are moved by a small amount. In the lemma below, we give a quantitative statement which will be used in the proof of Lemma \ref{lem:nested_cones_for_bounded_distances}.

\begin{lemma}
\label{lem:small_mobius_action} Let $m_0>0$ and $C>0$ be fixed. Define $\eta'(m_0,C)=\min\{\frac{1}{10(C+1)m_0},\frac{1}{4}\}$. If $0\leq\eta <\eta'$, $0<m\leq m_0$,  $A\in PSL_2(\mathbb{C})$ with
$$
\| A-Id\|<\eta m
$$
and $(p,u)\in T^1\mathbb{H}^3$ such that 
$$
D_{T^1\mathbb{H}^3}((p,u),(e^{-m}j,-j))<C\eta m
$$
then
$$
D_{T^1\mathbb{H}^3}(A(p,u),(e^{-m}j,-j))< C_1\eta m
$$
where $C_1=2\pi (60C+9)$.
\end{lemma}

\begin{proof}
Let $ht(p)=h$ and $Z(p)=z$. Let $p_1=A(p)$, and $ht(p_1)=h_1$ and $Z(p_1)=z_1$. For $A\in PSL_2(\mathbb{C})$, the Poincar\'e extension (cf. \cite{Bear}) of $A$ is given by the formula
$$
A(P)=\frac{(az+b)\overline{(cz+d)}+a\bar{c}h^2}{|cz+d|^2+|c|^2h^2}+\frac{h}{|cz+d|^2+|c|^2h^2}j.
$$

By the assumption, $|Z(p)-Z(e^{-m}j)|=|Z(p)|=|z|<C\eta m$ and $|ht(p)-ht(e^{-m}j)|=|h-e^{-m}|<C\eta m$. 
We set
$$
\eta'(m_0,C)=\min\big{\{}\frac{1}{10(C+1)m_0},\frac{1}{4}\big{\}}.
$$
If $\eta\leq\eta'(m_0,C)$ then $C\eta m\leq\frac{1}{10}$ and $\eta m\leq\frac{1}{10}$. Moreover, $\| A-Id\|<\eta m$ implies that $|b|,|c|,|a-1|,|d-1|<\eta m$. 

For $p_1=A(p)$, the above inequalities together with some elementary computations give
$$
|ht(p_1)-e^{-m}|\leq (2C+1)\eta m
$$
and
$$
|Z(p_1)|\leq 4(C+1)\eta m.
$$

Let $v=v_1+v_2i-j$ and $w=A'(p)v=w_1+w_2i+w_3j$. The assumption $|\angle (v,-j)|<C\eta m$ implies that $|v_1|,|v_2|<\frac{\pi}{2}C\eta m$. 
Note that
$$
w_k=\frac{\partial A_k}{\partial x}v_1+\frac{\partial A_k}{\partial y}v_2+ \frac{\partial A_k}{\partial h}v_3
$$
for $k=1,2,3$, where $A=A_1+A_2i+A_3j$. 

Some more elementary (and long) computations give
$$
|\frac{\partial A_1}{\partial x}|,|\frac{\partial A_1}{\partial y}|,|\frac{\partial A_2}{\partial x}|,|\frac{\partial A_2}{\partial y}|\leq 15,
$$

$$
|\frac{\partial A_1}{\partial h}|,|\frac{\partial A_2}{\partial h}|\leq 9\eta m,
$$

$$
|\frac{\partial A_3}{\partial x}|,|\frac{\partial A_3}{\partial y}|\leq 12\eta m,
$$
and
$$
|\frac{\partial A_3}{\partial h}|\geq 1-6\eta m
$$ 
at the point $p$.

The above estimates provide
$$
|w_1|,|w_2|\leq (60C+9)\eta m
$$
and
$$
|w_3|\geq 1-(4C+6)\eta m.
$$
Then
\begin{equation*}
|\angle (w,-j)|\leq\frac{\pi (60C+9)}{1-(4C+6)\eta m}\eta m.
\end{equation*}
Since $\eta\leq \frac{1}{(4C+6)m_0}$ we obtain
\begin{equation}
\label{eq:angle_phi}
|\angle (w,-j)|\leq{2\pi (60C+9)}\eta m.
\end{equation}

\end{proof}

Consider cone $\mathcal{C} (j,-j,\frac{\pi}{2})$ with the vertex $j$ and the central 
geodesic ray in the direction of the tangent vector $-j$. Then cone $\mathcal{C}
(e^{-m}j,-j,\frac{\pi}{2})$ is a subset of the above cone. If $(p,u)$ is close enough to $
(e^{-m}j,-j)$ and if $A\in PSL_2(\mathbb{R})$ is close enough to the identity, then the 
cone $\mathcal{C}(A(p),A'(p)u,\frac{\pi}{2})$ is a subset of $\mathcal{C} (j,-j,\frac{\pi}
{2})$.

The following lemma establishes a sufficient quantitative bound on the size of a hyperbolic isometry $A\in PSL_2(\mathbb{C})$ and on the distance of $(p,u)$ to $(e^{-m}j,-j)$ such that $\mathcal{C} (j,-j,\frac{\pi}{2})\supset \mathcal{C}(A(p),A'(p)u,\frac{\pi}{2})$.

\begin{lemma}
\label{lem:nested_cones_for_bounded_distances} Fix $m_0>0$ and $C>0$. Let $\eta''(m_0,C)=\min\{ \frac{e^{-m_0}}{32(C+1)m_0},\frac{e^{-m_0}}{66\pi (60C+9)}\}$. Then  for each $0\leq \eta<\eta''(m_0,C)$, for each $0<m\leq m_0$ and for each $A\in PSL_2(\mathbb{C})$ with
$$
\| A-Id\|<\eta m
$$
we have
$$
\partial_{\infty}\mathcal{C} (j,-j,\frac{\pi}{2})\supset \overline{\partial_{\infty}\mathcal{C}(A(p),A'(p)u,\frac{\pi}{2})},
$$
where $(p,u)\in T^1\mathbb{H}^3$ is such that 
$$
D_{T^1\mathbb{H}^3}((p,u),(e^{-m}j,-j))<C\eta m.
$$
\end{lemma}

\begin{proof}
Let $p_1=A(p)$;
write
$h_1=ht(P_1)$ and $z_1=Z(P_1)$. By the proof of Lemma \ref{lem:small_mobius_action} we have $$e^{-m}-(2C+1)\eta m\leq h_1\leq e^{-m}+(2C+1)\eta m$$ and $$|z_1|\leq 4(C+1)\eta m.$$  

Let $y=|z_1|$. Note that $z_1$ is the foot in $\mathbb{C}$ of the vertical line through $p_1$ perpendicular to $\mathbb{C}$. Let $x\in\mathbb{C}$ be the center of the euclidean hemisphere in $\mathbb{H}^3$ which passes through $p_1$ and is tangent to the unit radius euclidean hemisphere centered at $0\in\mathbb{C}$. Let $\varphi$ be the angle at $p_1$ between the radius of the hemisphere centered at $x$ and the vertical line through $p_1$ (cf. Figure 1).  

\begin{figure}
\centering
\includegraphics[scale=0.5]{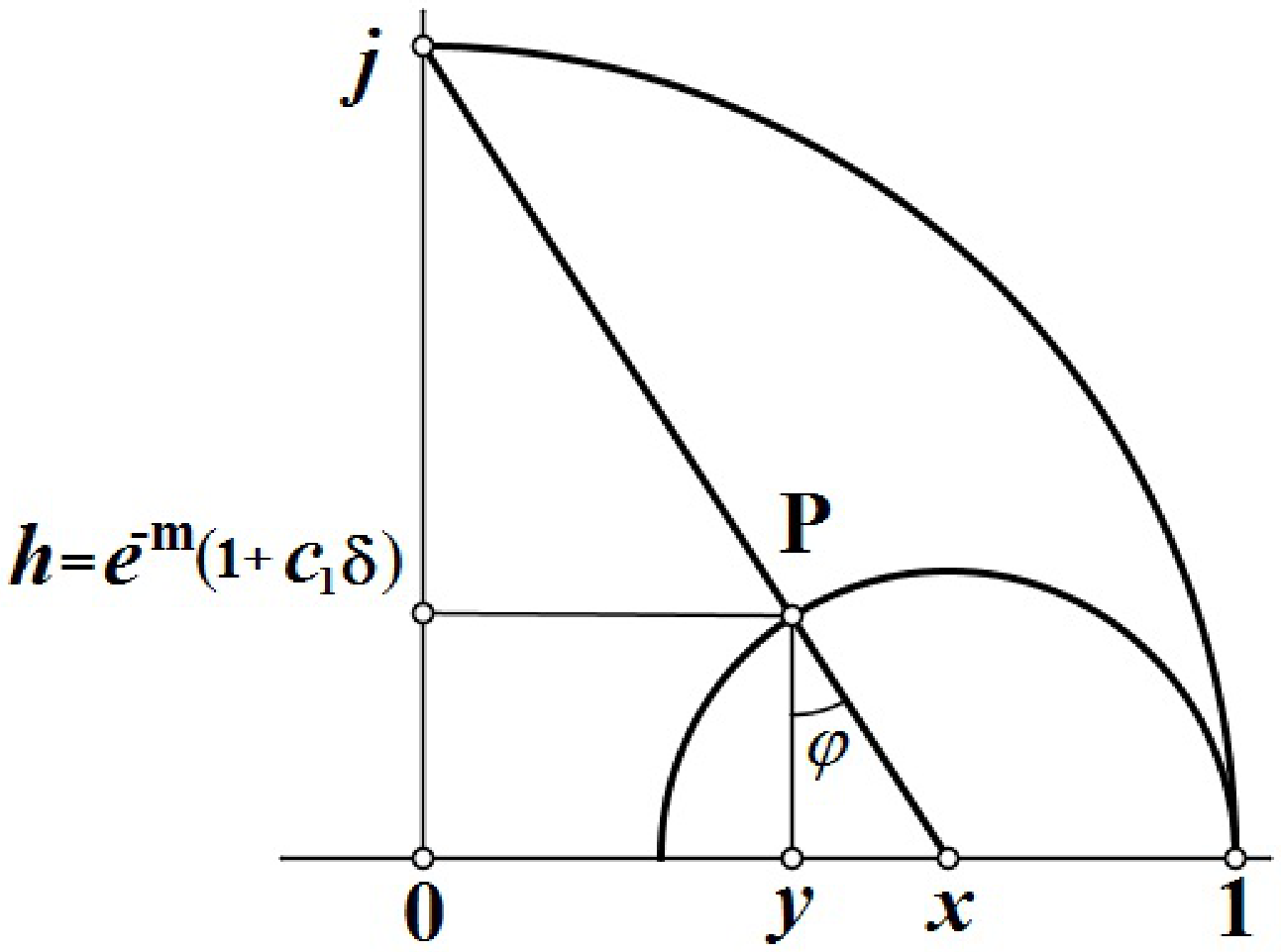}
\caption{}
\end{figure}

From Figure 1, we have
$$
(1-x)^2=(x-y)^2+h_1^2\leq x^2+h_1^2
$$
which implies
$$
x\geq \frac{1-h_1^2}{2}\geq\frac{1-h_1}{2}\geq \frac{1-e^{-m}-(2C+1)\eta m}{2}\geq \frac{1- (2C+1) \eta}{2}m.
$$
If $\eta$ satisfies
$$
\eta\leq\frac{1}{2(2C+1)},
$$
then
\begin{equation}
\label{eq:x-estimate}
x\geq\frac{1}{4}m.
\end{equation}
This gives
$$
\tan\varphi =\frac{x-y}{h_1}\geq \frac{\frac{1}{4}m-4(C+1)\eta m}{e^{-m}+(2C+1)\eta m}\geq \frac{1}{16}m
$$
when $\eta\leq\frac{1}{32m_0(C+1)}$. 

Let $\varphi_0$ be the angle at $p_1$ in Figure 1 when $m=m_0$. Then we have
\begin{equation}
\label{eq:varphi_estimate}
\varphi\geq\frac{\varphi_0}{\tan \varphi_0}\tan\varphi\geq \frac{\varphi_0}{\tan\varphi_0} \frac{1}{16}m.
\end{equation}
We need an upper bound on $\frac{\tan\varphi_0}{\varphi_0}$. Let $x_0$ and $y_0$ be the values of $x$ and $y$ when $m=m_0$. Similar to the above,
we have
$$
(1-x_0)^2=(x_0-y_0)^2+h^2\geq (x_0-y_0)^2
$$
which gives
$$
x_0\leq\frac{1+y_0}{2}\leq\frac{11}{20}
$$
because $y_0\leq C\eta m\leq\frac{1}{10}$. 
 Then
 $$
 \tan\varphi_0=\frac{x_0-y_0}{h}\leq\frac{x_0}{h}\leq\frac{\frac{11}{20}}{e^{-m_0}-(2C+1)\eta m_0}\leq\frac{11}{10}e^{m_0}
 $$
 for $\eta \leq\frac{e^{-m_0}}{2m_0(2C+1)}$. Since
 $$
 \frac{\tan\varphi_0}{\varphi_0}\leq\frac{1}{\cos\varphi_0}=\sqrt{\tan^2\varphi_0+1}\leq 2e^{m_0}
 $$
 by (\ref{eq:varphi_estimate}) we obtain
 \begin{equation}
 \label{eq:phi_main_estimate}
 \varphi\geq \frac{e^{-m_0}}{32}m.
 \end{equation}
 Let $\varphi_0^{*}=\frac{e^{-m_0}}{33}m$.
It follows that
$$
\overline{\partial_{\infty}\mathcal{C} (p_1,-j,\frac{\pi}{2}+\varphi_0^{*}m)}\subset \partial_{\infty}\mathcal{C}(j,-j,\frac{\pi}{2}).
$$
By Lemma \ref{lem:small_mobius_action}, $|\angle (w,-j)|\leq 2\pi (60C+9)\eta m$ when $0\leq\eta\leq\eta'(m_0,C)$. Thus we get
$$
\overline{\partial_{\infty}\mathcal{C} (p_1,w,\frac{\pi}{2})}\subset \partial_{\infty}\mathcal{C} (p_1,-j,\frac{\pi}{2}+\varphi_0^{*}m)
$$
when $\eta\leq\frac{e^{-m_0}}{66\pi (60C+9)}$ which implies the desired nesting of the cones. By putting the bounds on $\eta$ together, we define $\eta'' =\min\{ \frac{e^{-m_0}}{32(C+1)m_0},\frac{e^{-m_0}}{66\pi (60C+9)}\}$.
\end{proof}

The above lemma established nesting of a cone and the image of another cone under an isometry $A$ when the isometry is close enough to the identity. In the following lemma the only type of an isometry that we consider is a hyperbolic rotation of $\mathbb{H}^3$ around a geodesic $h$ intersecting the central geodesic ray of the outside cone. Unlike above, the rotation is not necessarily close to the identity since we allow the rotation angle to be arbitrary. The additional condition on the rotation is that its axis $h$ subtends a small angle with the central axis of the outside cone which implies the nesting.

\begin{lemma}
\label{lem:rotation_large_angle}
Let $m_0>0$ be fixed. Let $g$ be the positive $z$-axis in $\mathbb{H}^3$ and let $h$ be the geodesic in the $xz$-plane that intersects $g$ at the point $e^{-m}j\in\mathbb{H}^3$ subtending the angle $\theta$. Then for any $m$ with $0\leq m\leq m_0$, any $\theta$ with $0\leq\theta <\frac{e^{-m_0}}{16}$, and any $(p,u)\in T^1\mathbb{H}^3$ with
$$
D_{T^1\mathbb{H}^3}((p,u),(e^{-m}j,-j))=\delta <\frac{1}{4}
$$
we have
$$
D_{T^1\mathbb{H}^3}(R_h^{\varphi}(p,u),(e^{-m}j,-j))\leq 20\delta +40\sqrt{2} e^{m_0}\theta
$$
for any $\varphi\in\mathbb{R}$, where $R_h^{\varphi}$ is the rotation of $\mathbb{H}^3$ with the angle $\varphi$ around the axis $h$.
\end{lemma}

\begin{proof}
Let $b<0$ and $a>0$ be the endpoints of $h$. Since the angle between $g$ and $h$ is $\theta$, it follows $a=-b\tan^2\theta$. Moreover $h$ intersects $g$ in the geodesic arc $[j,e^{-m_0}]$ which gives $e^{-m_0}\leq\sqrt{-2ab}\leq 1$. Consequently
\begin{equation}
\label{eq:b}
\sqrt{2}e^{-m_0}\frac{1}{\theta}\leq |b|\leq\frac{\pi}{\sqrt{2}}\frac{1}{\theta}
\end{equation}
and 
\begin{equation}
\label{eq:a}
\frac{e^{-2m_0}}{\sqrt{2}\pi}\theta\leq a\leq\frac{e^{m_0}}{2\sqrt{2}}\theta .
\end{equation}
Note that
$$
R_h^{\varphi}(z)=\frac{\frac{a-e^{i\varphi}b}{e^{i\varphi/2}(a-b)}z+\frac{ab(e^{i\varphi}-1)}{e^{i\varphi /2}(a-b)}}{\frac{1-e^{i\varphi}}{e^{i\varphi /2}(a-b)}z+\frac{e^{i\varphi}a-b}{e^{i\varphi /2}(a-b)}}
$$
and let $R_h^{\varphi}(z+tj)$ be the extension of $R_h^{\varphi}$ to $\mathbb{H}^3$.
Then
\begin{equation*}
\begin{split}
Z(R_h^{\varphi}(z+tj))=\frac{\big{[}\frac{a-e^{i\varphi}b}{e^{i\varphi/2}(a-b)}z+\frac{ab(e^{i\varphi}-1)}{e^{i\varphi /2}(a-b)}\big{]}\overline{\big{[}\frac{1-e^{i\varphi}}{e^{i\varphi /2}(a-b)}z+\frac{e^{i\varphi}a-b}{e^{i\varphi /2}(a-b)}\big{]}}}{\big{|}\frac{1-e^{i\varphi}}{e^{i\varphi /2}(a-b)}z+\frac{e^{i\varphi}a-b}{e^{i\varphi /2}(a-b)}\big{|}^2+\big{|}\frac{1-e^{i\varphi}}{e^{i\varphi /2}(a-b)}\big{|}^2t^2}+\\
+\frac{\frac{a-e^{i\varphi}b}{e^{i\varphi/2}(a-b)}\frac{ab(e^{-i\varphi}-1)}{e^{-i\varphi /2}(a-b)}t^2}{\big{|}\frac{1-e^{i\varphi}}{e^{i\varphi /2}(a-b)}z+\frac{e^{i\varphi}a-b}{e^{i\varphi /2}(a-b)}\big{|}^2+\big{|}\frac{1-e^{i\varphi}}{e^{i\varphi /2}(a-b)}\big{|}^2t^2}
\end{split}
\end{equation*}
and
\begin{equation*}
ht(R_h^{\theta}(p,u))=\frac{t}{\big{|}\frac{1-e^{i\varphi}}{e^{i\varphi /2}(a-b)}z+\frac{e^{i\varphi}a-b}{e^{i\varphi /2}(a-b)}\big{|}^2+\big{|}\frac{1-e^{i\varphi}}{e^{i\varphi /2}(a-b)}\big{|}^2t^2}.
\end{equation*}
The bounds on $\delta$ and $\theta$ give that $a<\frac{1}{8}$, $|z|<\frac{1}{4}$ and $|b|>4$. 
Similar to the proof of Lemma \ref{lem:small_mobius_action} we obtain the desired estimates.
\end{proof}

In the following lemma we estimate the size of a hyperbolic rotation in terms of the distance of its axis to $j\in\mathbb{H}^3$ and the rotation angle.

\begin{lemma}
\label{lem:bound_on_norm}
Let $A$ be a rotation in $\mathbb{H}^3$ by an angle $\epsilon$ around geodesic $l_{a,b}$ with endpoints $a,b\in\bar{\mathbb{R}}=\partial_{\infty}\mathbb{H}^2\subset\partial_{\infty}\mathbb{H}^3$. If the geodesic with endpoints $a$ and $b$ intersects the ball of radius $m_0>0$ centered at $j\in\mathbb{H}^3$, then 
$$
\| A -Id\|\leq (1+e^{2m_0})\frac{|\epsilon |}{2}.
$$
\end{lemma}

\begin{proof}
We prove the lemma when both $a$ and $b$ are finite points. When $a$ or $b$ is $\infty$, the proof is left to the reader.

Note that 
$$
A(z)=\frac{\frac{a-be^{i\epsilon}}{(a-b)e^{i\frac{\epsilon}{2}}}z+\frac{ab(e^{i\epsilon}-1)}{(a-b)e^{i\frac{\epsilon}{2}}}}{\frac{1-e^{i\epsilon}}{(a-b)e^{i\frac{\epsilon}{2}}}z+\frac{ae^{i\epsilon}-b}{(a-b)e^{i\frac{\epsilon}{2}}}}.
$$

Assume that $|a|\leq |b|$.
Since $l_{a,b}$ intersects the ball of radius $m_0$ centered at $j\in\mathbb{H}^3$, it follows that
$|a|\leq e^{m_0}$ and $|a-b|\geq 2e^{m_0}$. Then
\begin{equation*}
\Big{|}\frac{a-be^{i\epsilon}}{(a-b)e^{i\frac{\epsilon}{2}}}-1\Big{|}=\frac{|1-e^{i\frac{\epsilon}{2}}|\cdot |a+be^{i\frac{\epsilon}{2}}|}{|a-b|}\leq (1+\frac{2|a|}{|a-b|})\frac{\epsilon}{2}\leq (1+e^{2m_0})\frac{\epsilon}{2}.
\end{equation*}

Further, we claim that
\begin{equation*}
\Big{|}\frac{ab(e^{i\epsilon}-1)}{a-b}\Big{|}\leq\frac{|ab|}{|a-b|}\epsilon\leq\frac{e^{2m_0}}{4}\epsilon.
\end{equation*}
To see this, note that $j\in\mathbb{H}^3$ is on the distance at most $m_0$ from the geodesic $l_{a,b}$. Then \cite[formula (7.20.4)]{Bear} gives
$$
\sinh m_0\geq\frac{\cosh^2m_0+ab}{|(b-a)|cosh m_0}
$$
which implies the above. 
In a similar fashion, we obtain 
$$
\Big{|}\frac{ae^{i\epsilon}-b}{(a-b)e^{i\frac{\epsilon}{2}}}-1\Big{|}\leq (1+e^{2m_0})\frac{\epsilon}{2}
$$
and 
$$
\Big{|}\frac{1-e^{i\epsilon}}{(a-b)e^{i\frac{\epsilon}{2}}}\Big{|}\leq e^{m_0}\epsilon .
$$ 
\end{proof}

The following lemma is well-known \cite{Bon2} and we estimate the constant involved.

\begin{lemma}
\label{lem:const_in_norm}
Let $g_1$ and $g_2$ be geodesics in $\mathbb{H}^2\subset\mathbb{H}^3$ that have a common endpoint and that intersect the ball $B_{m_0}(i)$ of radius $m_0>0$ centered at $i\in\mathbb{H}^2$. Let $s$ be a geodesic arc that connects $g_1$ and $g_2$ inside $B_{m_0}(i)$. Then
$$
\|R_{g_1}^{\epsilon}\circ R_{g_2}^{-\epsilon}-Id\|\leq 2e^{m_0}(1+e^{m_0})|s|\cdot |\epsilon |,
$$
where $|s|$ is the hyperbolic length of the geodesic arc $s$.
\end{lemma}

\begin{proof}
Let $A(z)=\frac{\cos\theta z+\sin\theta}{-\sin\theta z+\cos\theta}$ be a rotation of $\mathbb{H}^2$ around $i$. Note that $\| A\|\leq \sqrt{2}$. For the given embedding of $\mathbb{H}^2$ into $\mathbb{H}^3$, we have that $i\in\mathbb{H}^2$ is identified with $j\in\mathbb{H}^3$ and that the mapping $A$ acts as a rotation in $\mathbb{H}^3$ around the geodesic with endpoints $i,-i\in\partial_{\infty}\mathbb{H}^3=\bar{\mathbb{C}}$. The geodesic of the rotation passes through $j\in\mathbb{H}^3$ and $\mathbb{H}^2$ is orthogonal to this geodesic(thus setwise preserved by $A$). 

Let $t\in\mathbb{R}\cup \{\infty\}$ be the common endpoint of the geodesics $g_1$ and $g_2$. We choose the rotation $A$ such that $A(t)=\infty$. Then $A\circ R_{g_j}^{\epsilon}\circ A^{-1}=R_{g_j'}^{\epsilon}$, where $g_j'=A(g_j)$ for $j=1,2$. 
Note that
$$
\|R_{g_1}^{\epsilon}\circ R_{g_2}^{-\epsilon}-Id\|\leq\| A^{-1}\|\cdot\| R_{g_1'}^{\epsilon}\circ R_{g_2'}^{-\epsilon}-Id\|\cdot \| A\|=2\| R_{g_1'}^{\epsilon}\circ R_{g_2'}^{-\epsilon}-Id\|.
$$

Let $a$ and $b$ be the endpoints of $g_1'$ and $g_2'$, respectively.
 A short computation gives
$$
\| R_{g_1'}^{\epsilon}\circ R_{g_2'}^{-\epsilon}-Id\|=2|a-b|\cdot |\sin\frac{\epsilon}{2} |\leq |a-b| \cdot |\epsilon |.
$$
Let $P_1\in g_1'$ and $P_2\in g_2'$ be the endpoints of $s$. Without loss of generality, assume that $ht(P_1)\geq ht(P_2)$. Let $l'$ be the arc issued from $P_2$ that is orthogonal to $g_1'$. 
Let $x=ht(P_2)$. A direct computations gives
\begin{equation}
|l'|=\log \Big{[}\frac{|a-b|}{x}+\sqrt{1+\Big{(}\frac{|a-b|}{x}\Big{)}^2}\Big{]}
\end{equation}

Note that $h=\frac{|a-b|}{x}$ is the length of the horocyclic arc centered at $\infty$ between $g_1'$ and $g_2'$ at the height $x$. Let $h_0$ be the maximum of the lengths of horocyclic arcs with the center at $\infty$ and inside the ball of radius $m_0$ centered at $i\in\mathbb{H}^2$. Then we obtain
$$
|l'|\geq\log (1+h)\geq\frac{1}{1+h_0}h
$$
which implies 
$$
|a-b|\leq x(1+h_0)|l'|\leq e^{m_0}(1+e^{m_0})|l'|
$$
and the lemma follows. 
\end{proof}

\section{Injectivity on the boundary}

Recall that the embedding of $\mathbb{H}^2$ in $\mathbb{H}^3$ given by
the mapping $z=x+yi\mapsto x+yj$, for $y>0$ and $x\in\mathbb{R}$. In this section, we
prove that the bending map of a pleated surface realizing a
transverse cocycle $\beta\in\mathcal{H}(\lambda ,\mathbb{R}/2\pi
\mathbb{Z})$ induces an injective map from
$\partial_{\infty}\mathbb{H}^2$ into $\partial_{\infty}\mathbb{H}^3$
under geometric conditions on the bending transverse cocycle
$\beta$ given in Theorem \ref{thm:main}.

\vskip .2 cm

\subsection{Outline of the proof}

For any $x,y\in\partial_{\infty}\mathbb{H}^2$ with $x\neq y$ we need to prove that $\tilde{f}(x)\neq \tilde{f}(y)$. Let $g$ be a hyperbolic geodesic in $\mathbb{H}^2\subset\mathbb{H}^3$ whose endpoints are $x$ and $y$. If $g$ is a leaf of $\tilde{\lambda}$ then $\tilde{f}(g)\subset \mathbb{H}^3$ is a geodesic. Thus $\tilde{f}(x)\neq\tilde{f}(y)$ in the case when $g$ is a leaf of $\tilde{\lambda}$. 

The main work is in proving $\tilde{f}(x)\neq\tilde{f}(y)$ when $g$ is not a leaf of $\tilde{\lambda}$. Let $p$ be a point on $g$ contained in a plaque of $\tilde{\lambda}$. 
The geodesic $g$ is divided by $p$ into two geodesic rays $g_1$ and $g_2$. We form two hyperbolic cones $\mathcal{C}(p,g_1,\pi /2)$ and $\mathcal{C}(p,g_2,\pi /2)$ whose shadows on $\partial_{\infty}\mathbb{H}^3$ are disjoint and contain $x$ and $y$, respectively. The idea is to prove that $\tilde{f}(g_i)$ stays in the cone $\mathcal{C}(p,g_i,\pi /2)$, for $i=1,2$. It is enough to restrict ones attention to $\tilde{f}(g_1)$ and the proof for the other ray is analogous. 

We start by considering all the intersection points of $g_1$ with the boundary sides of the long rectangles of $\tilde{\tau}$. We form a division of $g_1$ into arcs $\{ (a_n,b_n)\}_n$ such that $a_n<b_n\leq a_{n+1}<b_{n+1}$ for all $a_n$, where $a_n,b_n$ are (some of the) points of the intersection of $g_1$ with the boundary sides of long rectangles and the arc $(b_n,a_{n+1})$, if $b_n<a_{n+1}$, is outside of the geometric train track $\tilde{\tau}$. Moreover, each $(a_n,b_n)$ is chosen such that 
either it connects two long sides of a long rectangle; or it connects a long side of one long rectangle with a long side of its immediate neighbor rectangle; or it connects one long side of a long rectangle with another long side of a rectangle while passing through another rectangle; or it connects a short side of one rectangle to the other short side of the same rectangle; or it connects a short side of one rectangle with a long side of the adjacent rectangle. In short, any arc $(a_n,b_n)$ intersects at most three long rectangles; and it either intersects the family of all geodesics crossing a long rectangle $E$ while connecting the long sides of $E$, or it intersects the family of arcs intersecting a subarc of a short side of rectangle $E$ while connecting the short sides of $E$.

We simplify the considerations by assuming that $(a_n,b_n)$ either connects two long sides or it connects two short sides of a single long rectangle $E$. The arguments for these two cases also prove the theorem in other case with slightly smaller $\epsilon$ and $\delta$. 

Assume first that $(a_n,b_n)$ connects two long sides of a long rectangle $E$. Let $P$ and $Q$ be plaques of $\tilde{\lambda}$ which contain $a_n$ and $b_n$. Recall that the realization $\varphi_{P,Q}$ is given by
$$ 
\varphi_{P,Q}=\psi_{P,Q}R_Q
$$
where
$$
\psi_{P,Q}=\lim_{\mathcal{P}\to\mathcal{P}_{P,Q}} B_1B_2\cdots B_n
$$
for $\mathcal{P} =\{ P_1,P_2,\ldots ,P_n\}$. 

Recall that $$B_{P_i}=R_{g_i^P}^{\beta (P,P_i)}R_{g_i^Q}^{-\beta
(P,P_i)}
$$
where $P_i\in\mathcal{P}$, $\beta (P,P_i)$ the $\beta$-mass of a geodesic arc connecting $P$ and $P_i$, $g_i^P$ the geodesic on the boundary of $P_i$ facing $P$, and $g_i^Q$ the geodesic on the boundary of $P_i$ facing $Q$. Moreover $$R_Q=R_{g_Q^P}^{\beta (P,Q)}.$$ 

 Using the fact that $\|\beta\|_{max}$ is small, we get that $R_Q$ is close to the identity by Lemma 5.4. Moreover, since $\|\beta\|_{var_{\delta}}$ is small when $\delta$ is fixed, a repeated use of Lemma 5.5 gives that $B_1B_2\cdots B_n$ is close to the identity independently of $n$ which implies that $\varphi_{P,Q}$ is close to the identity for $\epsilon >0$ and $\delta >0$ small enough. Then Lemma
 5.2 guarantees that the image (under normalized $\tilde{f}$ which is the identity on $P$) of the cone at $b_n$ is contained in the cone at $a_n$.
 
 Assume next that the arc $(a_n,b_n)$ connects two short sides of a long rectangle $E$. In this case $|\beta ([a_n,b_n])|$ might not be small. The above estimates on $\psi_{P,Q}$ still hold by the same method using the fact that  $\|\beta\|_{var_{\delta}}$ is small when $\delta$ is fixed and small enough. However, the rotation $R_Q$ might not have small angle which makes $R_Q$ bounded away from the identity. This is a new phenomenon which does not appear in the case when the bendings are by measured (i.e. countably additive) transverse cocycles. Since the arc $(a_n,b_n)$ connects two short sides of a long rectangle, it follows that the angles of the intersections of the arc $(a_n,b_n)$ with the leaves of $\tilde{\lambda}$ are small. Then Lemma 5.3 implies that the cone at $b_n$ is mapped by $R_Q$ to a nearby cone. Further, the nearby cone is mapped by $\psi_{P,Q}$ to a cone contained in the cone at $a_n$ by Lemma 5.2. This finishes the proof of the nesting of cones in both cases and the proof of the theorem.

\subsection{The first step in proof of Theorem \ref{thm:main}}
Let $\{ k_1,\ldots ,k_{n}\}$ be a set of geometric arcs for the geodesic lamination $\lambda$ satisfying all the properties given in \S 2. Let $\tau$ be the corresponding geometric train track. Let $\tilde{\lambda}$ and $\tilde{\tau}$ be lifts to $\mathbb{H}^2$ of the geodesic lamination $\lambda$ and the geometric train track $\tau$. For simplicity, we denote by $k_j$ any lift of an arc $k_j$ and by $E$ any lift of an edge $E$.

Let $g$ be an arbitrary geodesic in $\mathbb{H}^2$. Our goal is to
show that the endpoints of $g$ in $\partial_{\infty}\mathbb{H}^2$
are mapped to distinct points in $\partial_{\infty}\mathbb{H}^3$
under the bending map $\tilde{f}$ corresponding to the transverse
cocycle $\beta$. If $g$ is contained in $\tilde{\tau}$ then it
coincides with a geodesic of $\tilde{\lambda}$. Since the bending
map $\tilde{f}$ sends a geodesic in $\tilde{\lambda}$ to a geodesic
in $\partial_{\infty}\mathbb{H}^3$, it follows that the endpoints of
$g$ are mapped to distinct points.

The main case to consider is when $g$ transversely intersects 
$\tilde{\lambda}$. Then there exists $p\in g\cap
(\mathbb{H}^2-\tilde{\tau})$ because if $g$ is completely contained in $\tilde{\tau}$ then it is a geodesic of $\tilde{\lambda}$. The point $p$ divides the geodesic $g$
into two geodesic rays $g_1$ and $g_2$ with endpoints $x$ and $y$, respectively. We consider the geodesic ray
$g_1$ and similar conclusions hold for $g_2$. 

First divide the geodesic ray $g_1$ into subarcs using the points of
intersections of $g_1$ with the boundary sides of the edges of
$\tilde{\tau}$ (i.e. boundary sides of long rectangles) as
follows. The point $p$ is the initial point of $g_1$. The first
point $a_1$ of the intersection of $g_1$ with $\tilde{\tau}$ is at
a long side of an edge $E$ of $\tilde{\tau}$. We consider the
next point $p_1$ of the intersection of $g_1$ with a boundary side
of an edge $E$ of $\tilde{\tau}$. If $p_1$ is on the long side
of $E$ then we set $b_1=p_1$ and $[a_1,b_1]$ is the first subarc in
the division of $g_1$. If $p_1$ is on the short side of the edge
$E$ then we consider the next point $q_1$ of the intersection of
$g_1$ with boundary sides of the edges of $\tilde{\tau}$. If $q_1$
is on a long side of an edge then we set $b_1=q_1$. If $q_1$ is
on a short side and the next point of the intersection of $g_1$
with the boundary sides of $\tilde{\tau}$ is also on a short
side, then we set $q_1=b_1$. If $q_1$ is on a short side and the
next point of the intersection $r_1$ is on long side, then we
set $b_1=r_1$. Note that the arc $[a_1,b_1]$ intersects interiors of
at most three edges of $\tilde{\tau}$.

Assume that we have defined first $n$ arcs $\{ [a_1,b_1],\ldots
,[a_n,b_n]\}$ and we proceed to define $(n+1)$-st arc. If $b_n$ is
on the boundary of two edges of $\tilde{\tau}$, then we set
$a_{n+1}=b_n$; otherwise we let $a_{n+1}$ to be the first
intersection point of $g_1$ with the boundary sides of the edges of
$\tilde{\tau}$ that comes after $b_n$. Then $b_{n+1}$ is chosen in
a same fashion as $b_1$ above. We continue this process
indefinitely. Thus we obtain a family of arcs $\{
[a_n,b_n]\}_{n\in\mathbb{N}}$. If $a_n$ does not belong to a plaque
of $\tilde{\lambda}$ then we replace it with a nearby point on $g_1$
which belongs to a plaque and call the new point $a_n$ again. Do the
same for $b_n$. This situation occurs when $a_n$ or $b_n$ belong to the intersections of a short side of an edge of $\tilde{\tau}$ and the geodesic lamination $\tilde{\lambda}$. The
complement in $g_1$ of the union of arcs $[a_n,b_n]$ does not
intersect $\tilde{\lambda}$.

We consider a sequence of nested cones
$\mathcal{C}(a_n,g_1,\frac{\pi}{2}
)\supset\mathcal{C}(b_n,g_1,\frac{\pi}{2} )$ for $n\in\mathbb{N}$.
When the bending map $\tilde{f}$ is normalized to be the identity at the plaque containing $a_n$,
it is enough to prove that
$$
\overline{\partial_{\infty}\tilde{f}(\mathcal{C}(b_n,g_1,\frac{\pi}{2}))}\subset
\partial_{\infty}\mathcal{C}(a_n,g_1,\frac{\pi}{2}).
$$
Because this property is geometric, it is independent under the
post-compositions by the isometries of the bending map and it can be repeated along the
sequence of arcs $\{ [a_n,b_n]\}$. Thus we obtain that the sequence
of the images under the bending map of the shadows of the cones is
nested and in particular, the image under $\tilde{f}$ of the
endpoint of $g_1$ is contained in
$\partial_{\infty}\mathcal{C}(p,g_1,\frac{\pi}{2})$. Similarly, the
image under $\tilde{f}$ of the endpoint of $g_2$ is contained in
$\partial_{\infty}\mathcal{C}(p,g_2,\frac{\pi}{2})$. Since
$\partial_{\infty}\mathcal{C}(p,g_1,\frac{\pi}{2})\cap
\partial_{\infty}\mathcal{C}(p,g_2,\frac{\pi}{2})=\emptyset$, the bending
map $\tilde{f}$ sends the endpoints of $g$ into distinct points of
$\partial_{\infty}\mathbb{H}^3$.

To finish the proof, it remains to show that
\begin{equation}
\label{eqn:nesting_on_[a_n,b_n]}
\overline{\partial_{\infty}\tilde{f}(\mathcal{C}(b_n,g_1,\frac{\pi}{2}
))}\subset
\partial_{\infty}\mathcal{C}(a_n,g_1,\frac{\pi}{2} )
\end{equation}
where the bending map $\tilde{f}:\mathbb{H}^2\to\mathbb{H}^3$ is
normalized to be the identity at the point $a_n$. The rest of the
proof is divided into cases depending on the combinatorics of the
intersection of $[a_n,b_n]$ with the edges of $\tilde{\tau}$.

\subsection{Case I: $[a_n,b_n]$ connects two long sides of an edge $E$}

Recall that $l^{*}$ is the maximum
 of the diameters of the edges of the train track $\tilde{\tau}$.
Then each $[a_n,b_n]$ has length less than or equal to $3l^{*}$ since it
intersects at most three edges of $\tilde{\tau}$.

 Assume that $[a_n,b_n]$ intersects interior of a
single edge $E$ of $\tilde{\tau}$ and that it connects the two long
sides of $E$. Let $P$ and $Q$ be the plaques that contain $a_n$ and $b_n$,
respectively. By pre-composing the bending map with an isometry of $\mathbb{H}^2$, we can assume that $a_n=i\in\mathbb{H}^2$. 
Note that $\mathbb{H}^2$ is identified with $\{ (z,t): Im(z)=0,t>0\}\subset\mathbb{H}^3$ and under this identification $i\in\mathbb{H}^2$ corresponds to $j\in\mathbb{H}^3$.

Let $s(E)$ be a short side of $E$ (which means that it is a lift to $\mathbb{H}^2$ of an arc in $\{ k_1,\ldots ,k_n\}$, cf. \S 2). Then $|s(E)|\leq w^{*}\leq\frac{1}{20}$. Note that $s(E)$ is contained in ball of radius $l^{*}$ centered at $a_n=i\in\mathbb{H}^2$ because the diameter of $E$ is at most $l^{*}$. 

Recall that
$$
\psi_{P,Q}=\lim_{\mathcal{P}_l\to\mathcal{P}_{P,Q}}\psi_l
$$
where $\mathcal{P}=\{ P_1,P_2,\ldots ,P_n\}$ is a set of plaques between $P$ and $Q$ in the given order, 
$$
\psi_l=B_1B_2\ldots B_n$$
and
$$
B_i=R^{\beta (P,P_i)}_{g_{P_i}^P}R^{-\beta (P,P_i)}_{g_{P_i}^Q}
$$
where $g_{P_i}^P$ is the geodesic on the boundary of $P_i$ which separates $P_i$ from $P$; similar for $g_{P_i}^Q$; and $R_g^a$ is a hyperbolic rotation with the axis $g$ and the rotation angle $a$. 

Since points of $s(E)$ are on the distance at most $l^{*}$ from $i\in\mathbb{H}^2$, it follows by Lemma \ref{lem:const_in_norm} that
\begin{equation}
\label{eq:norm_estimate}
\| B_i-Id\|\leq 2e^{l^{*}}(1+e^{l^{*}})|\beta (P,P_i)|\cdot |d_i|
\end{equation}
where $|d_i|$ is the length of the gap $d_i=E\cap P_i$ and $\beta (P,P_i)$ is taken to be in the interval $(-\pi ,\pi ]$. 

We estimate the norm of $\prod_{i=1}^l B_i$ for arbitrary $l$. By (\ref{eq:norm_estimate}), we have that
$$
\| \prod_{i=1}^l B_i\|\leq \prod_{i=1}^l \|B_i\|
\leq \prod_{d} (1+2\pi e^{l^{*}}(1+e^{l^{*}})  |d|)
$$
where the last product is over all gaps $d$ of $s(E)$ with respect to $\tilde{\lambda}$.
Then 
\begin{equation*}
\begin{split}
\log \prod_{i=1}^l \|B_i\|
\leq\sum_{d} \log (1+ 2\pi e^{l^{*}}(1+e^{l^{*}}) |d|)\leq \\ 
\leq\sum_d 2\pi e^{l^{*}}(1+e^{l^{*}}) |d|\leq 2\pi e^{l^{*}}(1+e^{l^{*}}) |s(E)|.
\end{split}
\end{equation*}
Since $|s(E)|\leq\frac{1}{20}$, we have
$$
\| \prod_{i=1}^l B_i\|\leq e^{2\pi e^{l^{*}}(1+e^{l^{*}}) |s(E)|}\leq C(l^{*})
$$
where 
\begin{equation}
\label{eq:const_C}
C(l^{*})=e^{e^{l^{*}}}.
\end{equation}

This implies
$$
\|\psi_{P,Q}-Id\|\leq C(l^{*})\sum_d \| B_d-Id\|,
$$
where  the sum is over all gaps $d$ of $s(E)$ (i.e. components of $s(E)\setminus\tilde{\lambda}$) except the two components which contain the endpoints of $s(E)$; $P_d$ is the plaque which contains $d$; and
$B_d=R_{g_{P_d}^P}^{\beta (P,P_d)}\circ R_{g_{P_d}^Q}^{-\beta (P,P_d)}.$ 

We divide $\sum_d \| B_d-Id\|$ into two sums as follows. The first sum $\sum'$ is over finitely many gaps $\{ d_i: i=1,\ldots ,{n_{s(E)}}\}$ of 
$s(E)$ used in the definition of $\|\beta\|_{var_{\delta},s(E)}$ and the second sum $\sum''$ is over the remaining (infinitely many) gaps of $s(E)$.

Each term of $\sum'$ corresponding to a gap $d$ of $s(E)$ is bounded from above by $2e^{l^{*}}(1+e^{l^{*}}) |d|\cdot\|\beta\|_{var_{\delta}}$ by (\ref{eq:norm_estimate}).
Thus $$\sum'\leq 2e^{l^{*}}(1+e^{l^{*}})|s(E)|\cdot \|\beta\|_{var_{\delta}}.$$ 

The term of the second sum $\sum''$ which corresponds to a gap $d$ of $s(E)$ is bounded by $2\pi e^{l^{*}}(1+e^{l^{*}})|d|$ by (\ref{eq:norm_estimate}). Since the sum of the lengths of all $d$ in $\sum''$ is less than $\delta |s(E)|$, we obtain $$\sum''\leq 2\pi e^{l^{*}}(1+e^{l^{*}})\delta |s(E)|.$$ 

Taking the two estimates together, we have
$$
\|\psi_{P,Q}-Id\|\leq C'(l^{*})(\|\beta\|_{var_{\delta}} +\delta ) |s(E)| 
$$
where
$$
C'(l^{*})=2\pi e^{l^{*}}(1+e^{l^{*}})e^{e^{2l^{*}}}.
$$

By Lemma \ref{lem:bound_on_norm} and by $\beta (P,Q)=\beta (s(E))$, we get
 that $$\| R_Q-Id\| \leq (1+e^{2l^{*}}) |\beta (s(E))|/2\leq \frac{1+e^{2l^{*}}}{2}\|\beta\|_{max}$$
 where $R_Q=R_{g_Q^P}^{\beta (P,Q)}$ as in the definition of $\varphi_{P,Q}$. 
 
 Thus $$\|   R_Q\|\leq 1+\frac{1+e^{2l^{*}}}{2}\|\beta\|_{max}\leq \frac{3+e^{2l^{*}}}{2}$$ if we restrict to $\beta$ with $\|\beta\|_{max}\leq 1$. 
 
 Consequently, we obtain
\begin{equation}
\label{eq:norm-id-bound}\begin{split}
\|\varphi_{P,Q}-Id\|\leq C'(l^{*}) \frac{3+e^{2l^{*}}}{2}(\|\beta\|_{var_{\delta}} +\delta )|s(E)|  +\frac{1+e^{2l^{*}}}{2}\|\beta\|_{max}\leq\\
\leq C''(l^{*})(\|\beta\|_{var_{\delta}} +\delta )|s(E)|  +\frac{1+e^{2l^{*}}}{2}\|\beta\|_{max}
\end{split}
\end{equation}
where 
\begin{equation*}
C''(l^{*})=C'(l^{*}) \frac{3+e^{2l^{*}}}{2}=\pi (1+e^{l^{*}})(3+e^{2l^{*}})e^{l^{*}+e^{2l^{*}}}.
\end{equation*}

The inequality (\ref{eq:norm-id-bound}) holds for both short sides $s_1(E)$ and $s_2(E)$ 
of the edge $E$. Without loss of generality we assume that $|s_1(E)|=\min\{ |s_1(E)|,|s_2(E)|\}$. Lemma \ref{lem:size_rectangles} implies that $|[a_n,b_n]|\geq \frac{1}{20e^{l^{*}}}|s_1(E)|$, where $|[a_n,b_n]|$ is the length of the hyperbolic arc $[a_n,b_n]$.
Let $\eta''(l^{*},0)$ be the constant from Lemma \ref{lem:nested_cones_for_bounded_distances} for the given $l^{*}$ and $C=0$.
Lemma \ref{lem:nested_cones_for_bounded_distances} implies the desired nesting of the cones if
 the right side of (\ref{eq:norm-id-bound}) is less than $\eta''(l^{*},0)|[a_n,b_n]|$. To achieve this, it is enough to set
$$\epsilon =\frac{1}{2}\min\{ \frac{\eta''(l^{*},0)}{60e^{l^{*}}C''(l^{*})} ,\frac{2\eta''(l^{*},0)}{3(1+e^{2l^{*}})}\} =\frac{\eta''(l^{*},0)}{120e^{l^{*}}C''(l^{*})}$$ and $$\delta =\frac{\eta''(l^{*},0)}{120e^{l^{*}}C''(l^{*})}$$ for $\eta''(l^{*},0)$ given by Lemma \ref{lem:nested_cones_for_bounded_distances}.
The nesting
of the cones at $a_n$ and $b_n$ is guaranteed by Lemma
\ref{lem:nested_cones_for_bounded_distances} because $|[a_n,b_n]|\leq l^{*}$.

\subsection{Case 2: $[a_n,b_n]$ connects long side of an edge to a long side of an adjacent edge.}
Assume that $[a_n,b_n]$ enters an edge $E_1$ through
a long side, then enters an edge $E_2$ through a short side
in common with $E_1$ and exists $E_2$ through a long side of
$E_2$. 

Since the train track $\tau$ is bivalent
we have that the set of geodesics of $\tilde{\lambda}$ which intersect the arc $[a_n,b_n]$ is either the set of geodesics which traverses the edge $E_1$ or the set of geodesics which traverses the edge $E_2$. For definiteness, assume that we are in the former case. 

Let $s(E_1)$ be the short side of $E_1$ that contains one short side of $E_2$ and let $c_n=[a_n,b_n]\cap s(E_1)$.
Normalize such that $a_n=i\in\mathbb{H}^2$. Let $s(E_1)^1$ and $s(E_1)^2$ be the two arcs obtained by dividing $s(E_1)$ with the point $c_n$ such that the arcs $[a_n,c_n]$ and $s(E_1)^1$ have endpoints on the same long side $l_1$ of $E_1$, and that the arcs $[c_n,b_n]$ and $s(E_1)^2$ have endpoints on the same long side $l_2$ of $E_2$. Let $h_1$ and $h_2$ be two arcs from $c_n$ orthogonal to $l_1$ and $l_2$, respectively. 

It is immediate that $|[a_n,c_n]|\geq |h_1|$ and $|[c_n,b_n]|\geq |h_2|$. The hyperbolic sine rule and the fact that $|s(E_1)|\leq\frac{1}{20}$ give
$|s(E_1)^1|\leq \cosh \frac{1}{20}|h_1|\leq \cosh \frac{1}{20}|[a_n,c_n]|$ and $|s(E_1)^2|\leq \cosh \frac{1}{20}|h_2|\leq \cosh \frac{1}{20}|[c_n,b_n]|$. We obtain
$$
|s(E_1)|\leq \cosh \frac{1}{20}|[a_n,b_n]|.
$$

Similar to Case 1, we get
$$
\|\varphi_{P,Q}-Id\|\leq C''(2l^{*})(\|\beta\|_{var_{\delta}} +\delta )|s(E_1)|  +\frac{1+e^{4l^{*}}}{2}\|\beta\|_{max}
$$
where we use $2l^{*}$ instead of $l^{*}$ because the diameter of $E_1\cup E_2$ is $2l^{*}$.

 The nesting of the cones at $a_n$ and $b_n$ follows as in Case 1 with the constants $$\epsilon =\delta=\frac{\eta''(2l^{*},0)}{120e^{2l^{*}}C''(2l^{*})}$$ for $\eta''(2l^{*},0)$ given by Lemma \ref{lem:nested_cones_for_bounded_distances}. The later case is dealt with in the same fashion with the same constants.

\subsection{Case 3: $[a_n,b_n]$ connects two short sides of an edge.}
Assume that $[a_n,b_n]$ enters a short side of an
edge $E_1$, then it enters a short side of an edge $E_2$ which
is in common with $E_1$ and it exists a long side of $E_2$. See
Figure 2 and Figure 3 for different possibilities of the relative
positions of $E_1$, $E_2$ and $g_1$. Let $P$ and $Q$ be the plaques
that contain $a_n$ and $b_n$, respectively.

\begin{figure}
\centering
\includegraphics[scale=0.5]{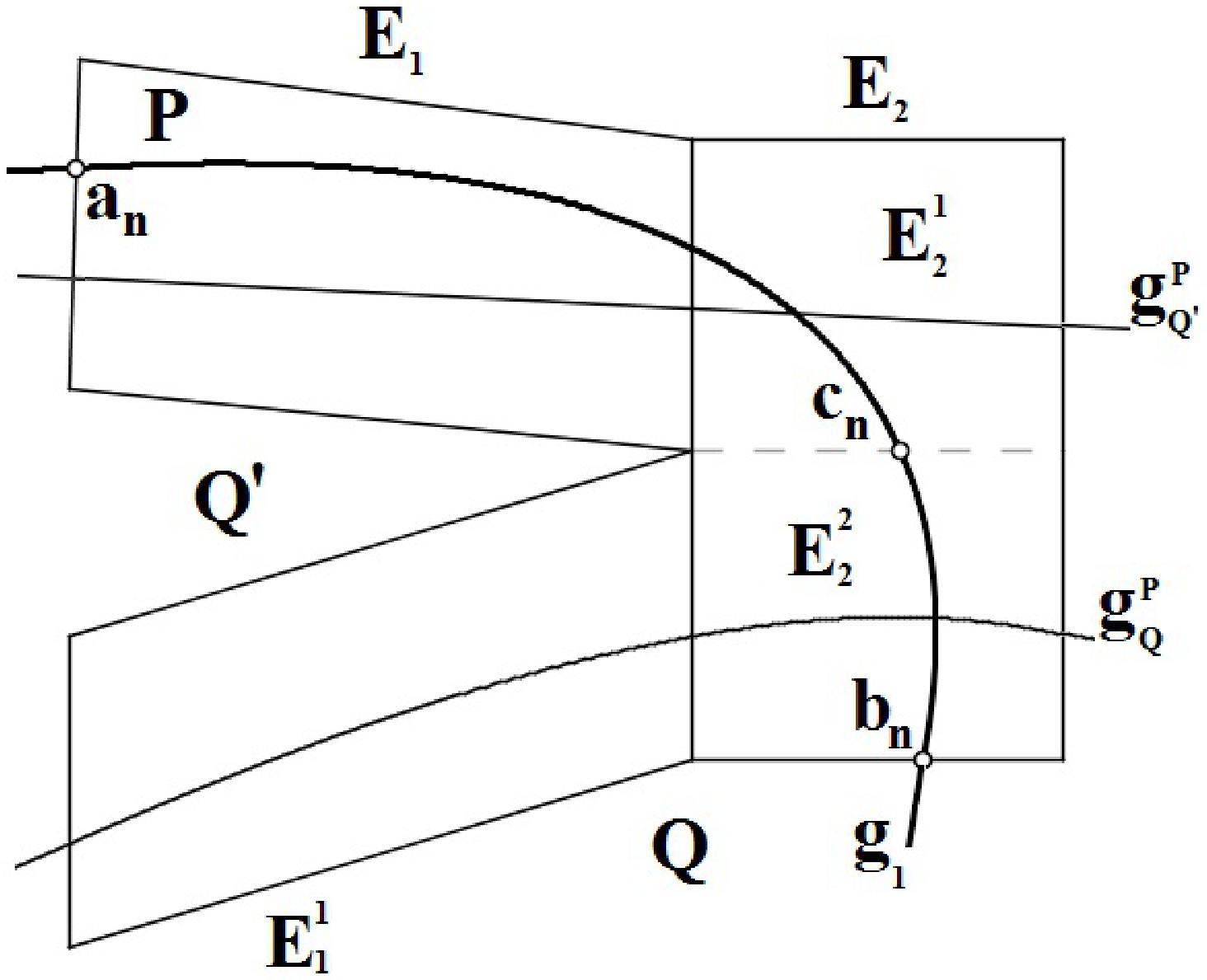}
\caption{}
\end{figure}

For the position in Figure 2 we argue as follows. Let $E_1^1$ be the incoming edge which meets $E_2$ at the same short side as $E_1$. Let $Q'$ be the
plaque that separates the geodesics of $\tilde{\lambda}$ that traverse $E_1$ from the geodesics of $\tilde{\lambda}$ that traverse
$E_1^1$. Then we have
$$
\varphi_{P,Q}= \varphi_{P,Q'}\circ \varphi_{Q',Q}.
$$
Note that $\beta (Q',Q)=\beta (E_1^1)$ which implies that
$|\beta (Q',Q)|\leq \|\beta\|_{max}$. 

By reasoning as in Case 1, we obtain 
$$
\|\varphi_{Q',Q}-Id\| \leq C''(2l^{*})(\|\beta\|_{var_{\delta}} +\delta )|s(E_1^1)|  +\frac{1+e^{4l^{*}}}{2}\|\beta\|_{max} 
$$
where $s(E_1^1)$ is the short side of $E_1^1$ contained in a short side of $E_2$.
Recall that we normalized such that $a_n=j$ and $b_n=e^{-m}j$ for $0\leq m\leq 2l^{*}$. Since $[a_n,b_n]$ connects the short sides of the edge $E_1$, it follows that $|[a_n,b_n]|\geq l_{*}>0$.

We apply Lemma \ref{lem:small_mobius_action} to $(e^{-m}j,-j)\in T^1\mathbb{H}^3$ with the constants $2l^{*}$, $C=0$ and $m\geq l_{*}>0$. We get
\begin{equation}
\label{eq:dist_QQ}
\begin{split}
D_{T^1\mathbb{H}^3}(\varphi_{Q',Q}(e^{-m}j,-j),(e^{-m}j,-j))\leq
18\pi [C''(2l^{*})(\|\beta\|_{var_{\delta}} +\delta )|s(E_1^1)|  +\\ 
+\frac{1+e^{4l^{*}}}{2}\|\beta\|_{max} ]<\eta' (2l^{*},0)\ \ \ \ \ \ 
\end{split}
\end{equation}
whenever 
\begin{equation}
\label{eq:cond1}
\frac{C''(2l^{*})(\|\beta\|_{var_{\delta}} +\delta )|s(E_1^1)|  +\frac{1+e^{4l^{*}}}{2}\|\beta\|_{max}}{l_{*}} <\frac{\eta' (2l^{*},0)}{18\pi}.
\end{equation}

Recall that
$$
\varphi_{P,Q'}=\psi_{P,Q'}\circ R_{g_{Q'}^P}^{\beta (P,Q')}
$$
where, in general, $\beta (P,Q')\neq\beta (E_1)$ since $a_n$ belongs to a short side of $E_1$. In fact $|\beta (E_1)|$ might not be even close to $0$.

Let $c_n$ be the point of the intersection between $[a_n,b_n]$ and the common boundary $s(E_1^1)$ of $E_1$ and $E_2$. It follows that the arc $[a_n,c_n]$ and the subarc of any geodesic of $\tilde{\lambda}$ that traverses $E_1$ are remaining close for the length $l_{*}$. This implies that they intersect at small angles. We give a numerical statement. 

\begin{lemma}
\label{lem:angle_of_intersection}
Let $E$ be an edge of a train track such that the shortest geodesic connecting the short sides has length at least $l>0$ and the maximum of the lengths of the short sides is at most $x$. Let $a_1$ and $a_2$  be geodesic arcs in $E$ connecting the short sides intersecting at an angle $\phi$. Then
$$
\phi \leq\frac{\pi}{2}(\coth\frac{l}{2})x.
$$
\end{lemma}

\begin{proof}
Let $A=a_1\cap a_2$. Then $A$ divides $a_1$ into two sub arcs $a_1'$ and $a_2'$. Without loss of generality, we can assume that the length of $a_1'$ is at least $\frac{1}{2}l$. Then the length of the subarc $a_2'$ of the arc $a_2$ which connects the short side of $E$ which contains an endpoint of $a_1'$ to the point $A$ is at least $\frac{1}{2}l-x$. Let $h$ be the geodesic arc issued from the endpoint of $a_1'$ on a short side of $E$ orthogonal to $a_2'$. The length of $h$ is at most $x$. We obtained a right angled triangle with one angle $\varphi$ whose opposite side has length at most $x$, and the side opposite to the right angle has the length at most $\frac{1}{2}l$. The hyperbolic sine rule gives
$$
\sin\phi =\frac{\sinh x}{\sinh\frac{1}{2}l}
$$
which implies 
$$
\phi\leq\frac{\pi}{2}(\coth\frac{l}{2})x.
$$
\end{proof}

We apply $R_{g_{Q'}^P}^{\beta (P,Q')}$ to $\varphi_{Q',Q}(e^{-m}j,-j)$. By Lemma \ref{lem:angle_of_intersection}, the angle of intersection $\phi$ between $g_{Q'}^P$ and $g_1$ satisfies
\begin{equation}
\label{eq:dist_phi}
\phi\leq\frac{\pi}{2}(\coth\frac{l_{*}}{2})w^{*}.
\end{equation}

By Lemma \ref{lem:rotation_large_angle}, 
we have
\begin{equation}
\label{eq:dist_RQQ}
\begin{split}
D_{T^1\mathbb{H}^3}([R_{g_{Q'}^P}^{\beta (P,Q')}\circ\varphi_{Q',Q}](e^{-m}j,-j),(e^{-m}j,-j))\leq\ \ \ \ \ \ \ \ \ \\ 20 D_{T^1\mathbb{H}^3}(\varphi_{Q',Q}(e^{-m}j,-j),(e^{-m}j,-j))+40\sqrt{2}e^{2l^{*}}\phi
\end{split}
\end{equation}
when $D_{T^1\mathbb{H}^3}(\varphi_{Q',Q}(e^{-m}j,-j),(e^{-m}j,-j))<\frac{1}{4}$ and $\phi<\frac{e^{-2l^{*}}}{16}$.  The former condition is satisfied because $D_{T^1\mathbb{H}^3}(\varphi_{Q',Q}(e^{-m}j,-j),(e^{-m}j,-j))<\eta'(2l^{*},0)\leq\frac{1}{4}$. To achieve the later condition we require that $\frac{\pi}{2}(\coth\frac{l_{*}}{2})w^{*}<\frac{e^{-2l^{*}}}{16}$ which implies that 
\begin{equation}
\label{eq:firstmaxk}w^{*}<\frac{e^{-2l^{*}}\tanh\frac{l_{*}}{2}}{8\pi}.
\end{equation}

By (\ref{eq:dist_QQ}), (\ref{eq:dist_phi}) and (\ref{eq:dist_RQQ}) we have
\begin{equation*}
\begin{split}
D_{T^1\mathbb{H}^3}([R_{g_{Q'}^P}^{\beta (P,Q')}\circ\varphi_{Q',Q}](e^{-m}j,-j),(e^{-m}j,-j))\leq\ \ \ \ \ \ \ \ \ \ \ \ \ \ \\ \frac{560\pi}{l_{*}} [C''(2l^{*})(\|\beta\|_{var_{\delta}}+\delta )|s(E_1^1)|+\frac{1+e^{4l^{*}}}{2}\|\beta\|_{max}]l_{*}+\\
+\frac{20\sqrt{2}\pi e^{2l^{*}}\coth\frac{l_{*}}{2}}{l_{*}}w^{*}l_{*}\ \ \ \ \ \ \ \ \ 
\end{split}
\end{equation*}

Let $\eta''(2l^{*},1)$ be the constant from Lemma \ref{lem:nested_cones_for_bounded_distances}. 
If 
\begin{equation}
\label{eq:delta_var}
\delta,\|\beta\|_{var_{\delta}}\leq \frac{l_{*}}{4\cdot 560\pi |k_1| C''(2l^{*})}\eta''(2l^{*},1),
\end{equation}
\begin{equation}
\label{eq:max}
\|\beta\|_{max}\leq\frac{l_{*}}{4\cdot 560\pi 3e^{4l^{*}}}\eta''(2l^{*},1)
\end{equation}
and 
\begin{equation}
\label{eq:maxk}
w^{*}\leq \frac{l_{*}}{20\sqrt{2}\pi e^{2l^{*}}\coth\frac{l_{*}}{2}}\eta''(2l^{*},1),
\end{equation}
then 
\begin{equation}
\label{eq:distance_on_tangent}
D_{T^1\mathbb{H}^3}([R_{g_{Q'}^P}^{\beta (P,Q')}\circ\varphi_{Q',Q}](e^{-m}j,-j),(e^{-m}j,-j))\leq\eta''(2l^{*},1)l_{*}.
\end{equation}

By Case 1, we immediately obtain the estimate 
$$\|\psi_{P,Q'}-Id\|\leq C'(2l^{*})(\|\beta\|_{var_{\delta}} +\delta ) |s(E_1^1)| .
$$ 
It follows that
$$
\|\psi_{P,Q'}-Id\|\leq \eta''(2l^{*},1)
$$
if 
$$
\delta ,\|\beta\|_{var_{\delta}}\leq\frac{l_{*}}{2C'(2l^{*})|s(E_1^1)|}
$$
which is satisfied because $C'(2l^{*})\leq C''(2l^{*})$ and by (\ref{eq:delta_var}). Therefore, Lemma \ref{lem:nested_cones_for_bounded_distances} and $w^{*}\leq 1/20$ implies the nesting of the cones if
$$
\epsilon =\frac{l_{*}\eta''(2l^{*},1)}{130\pi C''(2l^{*})} \leq\min\{ \frac{l_{*}\eta''(2l^{*},1)}{4\cdot 560\pi C''(2l^{*})|s(E_1^1)|},\frac{l_{*}\eta''(2l^{*},1)}{4\cdot 560\pi \frac{1+e^{4l^{*}}}{2}}\}
$$
and
$$
\delta\leq \frac{l_{*}\eta''(2l^{*},1)}{130\pi C''(2l^{*})}
$$
and (\ref{eq:maxk}) holds.

\vskip .1 cm

We consider the positions in Figure 3. The top left position in Figure 3 is a subcase of the position in Figure 3 where we set $\varphi_{Q'Q}=Id$ and the nesting follows for the same choices of $\epsilon$, $\delta$ and $w^{*}$. The top right position is exactly dealt as with the top left position. The bottom position in Figure 3 is exactly equal to the position in Figure 2 and the nesting is achieved by choosing the same constants.

\begin{figure}
\centering
\includegraphics[scale=0.3]{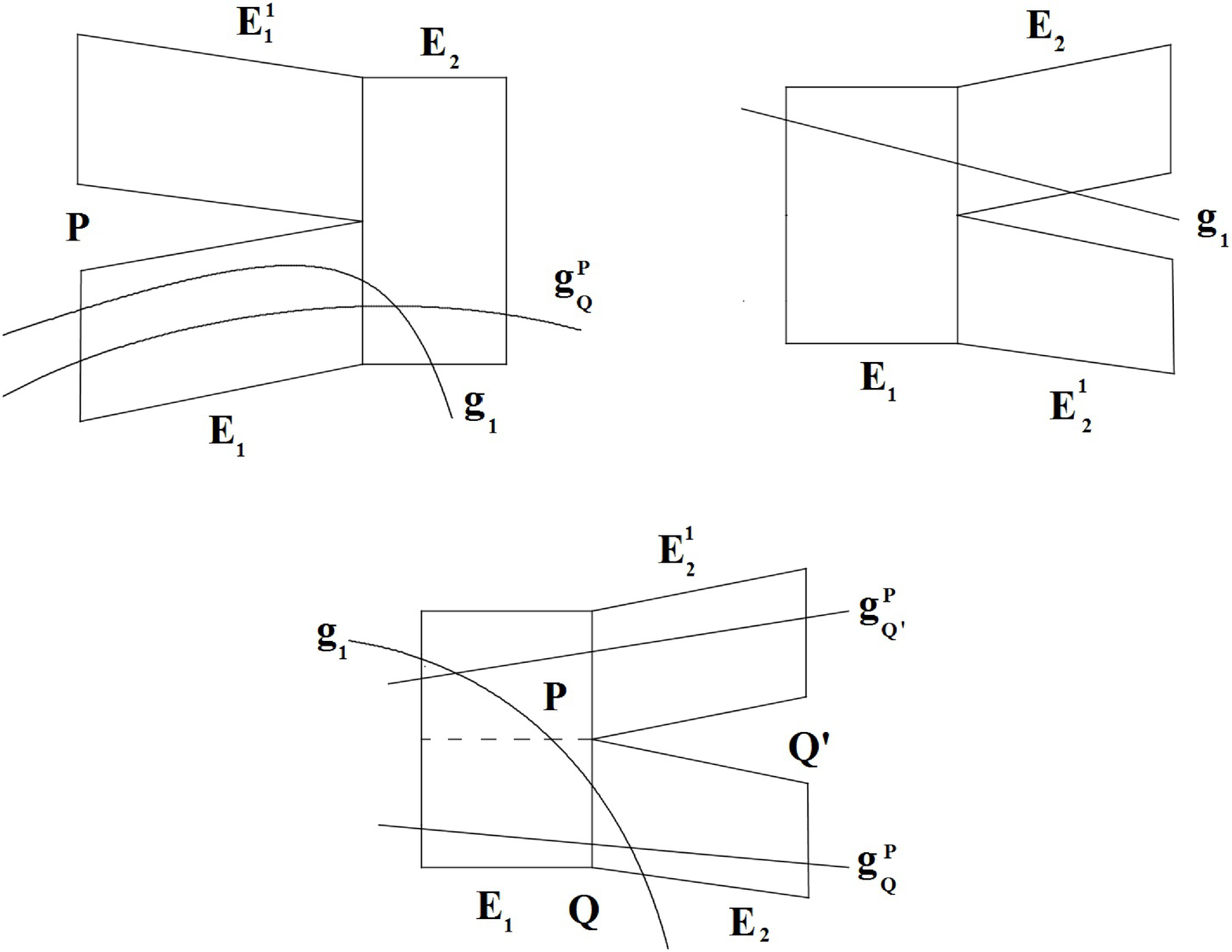}
\caption{}
\end{figure}

\subsection{All other cases for $[a_n,b_n]$.}

\hskip 2 cm

\vskip .2 cm

{\it Case 4.} Assume that $[a_n,b_n]$ enters $E$ on a short edge
and that it exists $E$ on the opposite short edge. The argument
in this case is contained in the second part of Case 3 and the bounds are the same.

\vskip .1 cm

{\it Case 5.} Assume that $[a_n,b_n]$ enters an edge $E_1$ of
$\tilde{\tau}$ through a long side, enters another edge $E_2$
through a short side in common with $E_1$ and then enters an
edge $E_3$ through a short side in common with $E_2$, and then
exists $E_3$ through a long side. Let $P$ and $Q$ be the plaques
of $\tilde{\tau}$ which contain $a_n$ and $b_n$, respectively. Note that the arc $[a_n,b_n]$ has length at least $m_{*}$ because it traverses the edge $E_2$.
Moreover, since the arc $[a_n,b_n]$ connects two long sides of different edges of $\tilde{\tau}$ it follows that the set of geodesics of $\tilde{\tau}$ that intersect $[a_n,b_n]$ is disjoint union of at most three sets of geodesics each of them traversing an edge of $\tilde{\tau}$. The situation in Figure 4 illustrates the case when this union consists of the geodesics traversing the edge above $E_2$, the edge $E_2$ and the edge below $E_2$. Note that the short sides of these three geodesics are on the distance at most $3l^{*}$ from $a_n=i$. 
Other possibilities can be easily checked by drawing pictures. It always happen that the set of geodesics of $\tilde{\lambda}$ intersecting $[a_n,b_n]$ is the disjoint union of at most three sets of geodesics traversing three edges of $\tilde{\tau}$ whose short sides are on the distance at most $3l^{*}$ from $a_n$. Therefore $\varphi_{P,Q}$ is the composition of at most three M\"obis maps $\varphi_{E_i'}$, for $i=1,2,3$, each corresponding to an edge $E_i'$ of $\tilde{\tau}$.

\begin{figure}
\centering
\includegraphics[scale=0.5]{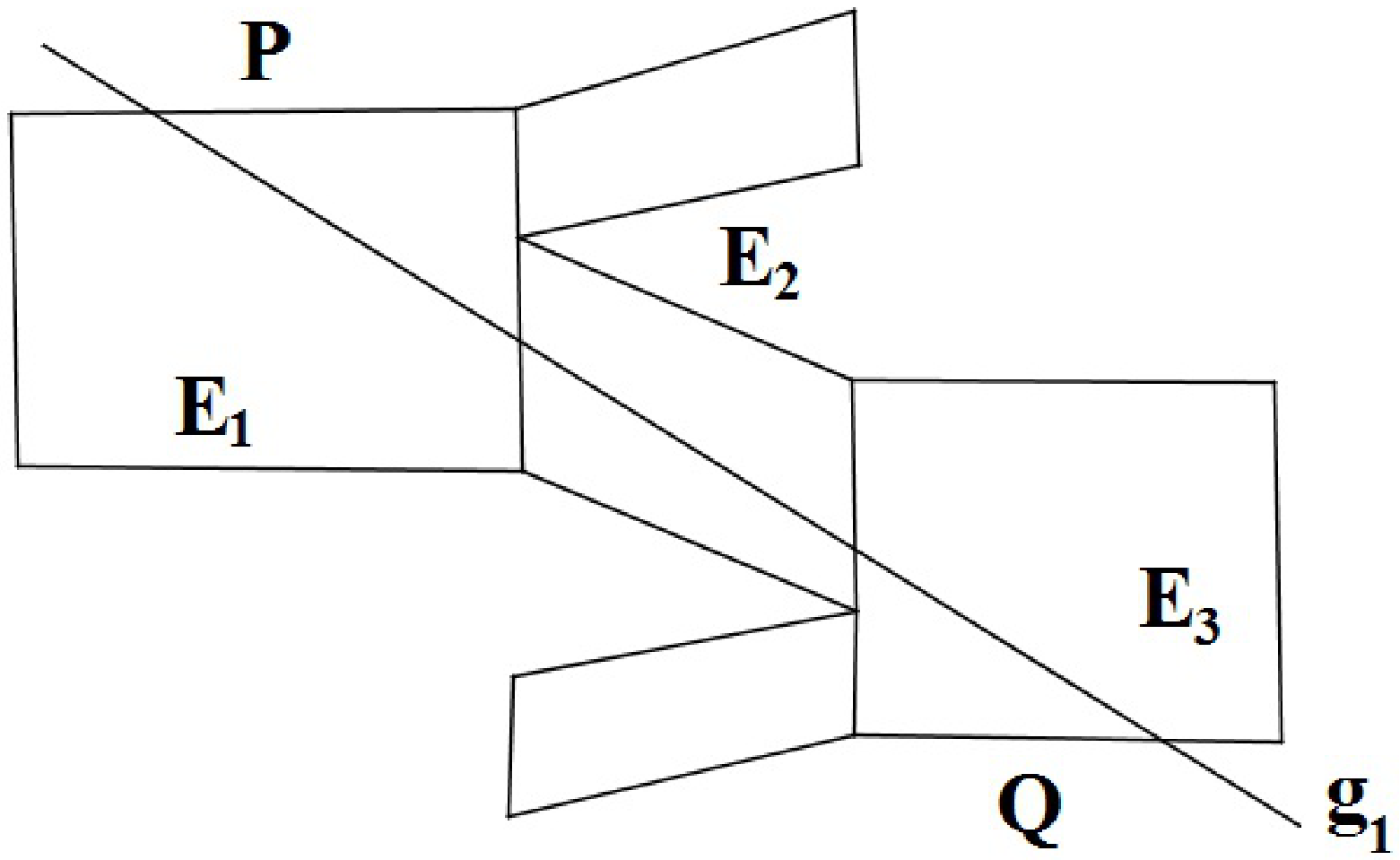}
\caption{}
\end{figure}

We use the argument from Case 1 to estimate $\|\varphi_{E_i'}-Id\|$. Namely, it is enough to replace $l^{*}$ with $3l^{*}$ to obtain
\begin{equation}
\label{eq:varphi_on_E}
\|\varphi_{E_i'}-Id\|\leq C''(3l^{*})(\|\beta\|_{var_{\delta}}+\delta )|s(E_i')|+\frac{1+e^{6l^{*}}}{2}\|\beta\|_{max}.
\end{equation}

Consequently, we have
\begin{equation}
\label{eq:est_varphi_product}
\begin{split}
\|\varphi_{P,Q}-Id\|\leq 3\big{(}1+ C''(3l^{*})(\|\beta\|_{var_{\delta}}+\delta )|s(E_i')|+\frac{1+e^{6l^{*}}}{2}\|\beta\|_{max}\big{)}^2 \times\\ \times \big{(}C''(3l^{*})(\|\beta\|_{var_{\delta}}+\delta )|s(E_i')|+\frac{1+e^{6l^{*}}}{2}\|\beta\|_{max}\big{)}.
\end{split}
\end{equation}

Assume that $\delta ,\|\beta\|_{var_{\delta}}\leq\frac{1}{2}$ and $\|\beta\|_{max}\leq 1$. Then $1+ C''(3l^{*})(\|\beta\|_{var_{\delta}}+\delta )|s(E_i')|+\frac{1+e^{6l^{*}}}{2}\|\beta\|_{max}\leq 1+C''(3l^{*})+\frac{1+e^{6l^{*}}}{2}$ because $|s(E_i')|\leq\frac{1}{20}$. We choose
$$
\epsilon =\delta = \frac{m_{*}\eta''(3l^{*},0)}{18C''(3l^{*})(1+C''(3l^{*})+\frac{1+e^{6l^{*}}}{2})^2}.
$$

\vskip .1 cm

\subsection{The end of the proof}
We established that the cones are nested along the sequence $\{
[a_n,b_n]\}_{n\in\mathbb{N}}$. Thus the bending map $\tilde{f}$ is
injective on $\partial_{\infty}\mathbb{H}^2$ as claimed.
We choose $\epsilon$ and $\delta$ to be the minimum over all cases and the nesting is guaranteed always.
{\it This ends the proof of Theorem 1.1}.

\vskip .3 cm

\begin{remark}
The size of $\epsilon$, $\delta$ and $w^{*}$ depends on the above constants $l^{*}$ and $l_{*}$ (cf. Table \ref{tab:consts}). 
The minimum $l_{*}$ of the distances between short sides of the edges of $\tilde{\tau}$ can be arbitrary small. In fact, when there are short closed geodesics contained in the geodesic lamination $\lambda$ then the train track $\tau$ cannot be modified such that $l_{*}$ is bigger than a universal positive constant. This fact forces us to include $l_{*}$ as a part of the geometric information for the geodesic lamination $\lambda$.

\begin{table}
\centering
\begin{tabular}{| c | c | c |}
\hline
$m_0$ & $\epsilon =\delta$ &  $w^{*}$  \\
\hline
$.000001$ & $2.20317\times 10^{-17}$ & $2.45816\times 10^{-20}$\\
\hline
$.00001$ & $2.20241\times 10^{-16}$ & $2.45807\times 10^{-18}$\\
\hline
$.0005$ &  $1.08066\times 10^{-14}$ & $6.13315\times 10^{-15}$ \\ \hline
$.001$ & $2.1201\times 10^{-14}$ & $2.44836\times 10^{-14}$  \\
\hline
$.0015$ & $3.1194\times 10^{-14}$ & $5.4978\times 10^{-14}$
\\ \hline
$.002$ & $4.07961\times 10^{-14}$ & $9.75434\times 10^{-14}$
\\ \hline 
$.0025$ & $5.00174\times 10^{-14}$ & $1.52107\times 10^{-13}$
\\ \hline 
$.003$ & $5.8868\times 10^{-14}$ & $2.18597\times 10^{-13}$
\\ \hline 
$.005$ & $9.07579\times 10^{-14}$ & $6.02374\times 10^{-13}$
\\ \hline 
$.01$ & $1.4901\times 10^{-13}$ & $2.36178\times 10^{-12}$
\\ \hline 
$.05$ & $1.33635\times 10^{-13}$ & $5.03139\times 10^{-11}$
\\ \hline 
$.1$ & $2.06663\times 10^{-14}$ & $1.64768\times 10^{-10}$
\\ \hline 
$.25$ & $3.41015\times 10^{-19}$ & $5.6501\times 10^{-10}$
\\ \hline 
$.5$ & $9.94507\times 10^{-43}$ & $8.30612\times 10^{-10}$
\\ \hline 
$1$ & $5.6123380\times 10^{-550}$ & $4.479\times 10^{-10}$\\
\hline
$2$ & $1.90389\times 10^{-212091}$ & $3.23146\times 10^{-11}$
\\
\hline
\end{tabular}
\label{tab:consts}
\vskip .2 cm
\caption{Values of $\epsilon$, $\delta$ and $k^{*}$ for the given $m_0$.}
\end{table}

If $\lambda$ does not contain closed geodesics then there exists a choice of a geometric train track $\tau$ which carries $\lambda$ such that $l_{*}\geq 1/20$ and $l^{*}= 1/5$. In this case we explicitly compute the constants in Theorem 1.1 to be
$$w^{*}=4.41719\times 10^{-10}$$ and
$$\epsilon = \delta=3.61749\times
10^{-17}.$$
\end{remark}

\section{Holomorphic motions and shear-bend cocycle}

Given a fixed hyperbolic surface $S$ and a maximal geodesic lamination $\lambda$ on $S$, we defined a geometric train track $\tau$ that carries $\lambda$ (cf. \S 6). Then we found universal $\epsilon >0$ and $\delta >0$ such that when an $(\mathbb{R}/2\pi\mathbb{Z})$-valued transverse cocycle $\beta$ satisfies $\|\beta\|_{max}<\epsilon w_{*}$ and $\|\beta\|_{var_{\delta}}<\epsilon$ then the bending map $\tilde{f}$ with the bending cocycle $\beta$ extends by continuity to an injection $\tilde{f}:\partial_{\infty}\mathbb{H}^2\to\partial_{\infty}\mathbb{H}^3$. 

We are considering holomorphic motions in this section, and injectivity of a family of maps is an essential part of the definition. 
To establish injectivity of a family of bending maps, we use the sufficient condition on $\beta$ obtained in the previous section. When the hyperbolic metric on $S$ is slightly changed, the metric quantities of the geometric train track $\tau$ are slightly changed. This fact is used to prove that there is an open neighborhood  of any  $\mathbb{R}$-valued transverse cocycle representing a hyperbolic metric on
$S$ in the space of $(\mathbb{C}/2\pi i\mathbb{Z})$-valued transverse cocycles whose points induce injective pleating maps.
 
\vskip .2 cm 
 
Let $K\subset\hat{\mathbb{C}}$ and let $\mathbb{D}=\{
w\in\mathbb{C}:|w|<1\}$. A {\it holomorphic motion} of a set $K$ is a map
$$
f:K\times\mathbb{D}\to\hat{\mathbb{C}}
$$
such that
$$
f(\cdot ,w ):K\to\hat{\mathbb{C}}
$$
is injective for each $w\in\mathbb{D}$, $f(z,0)=z$ for all $z\in
K$, and
$$
f(z,\cdot ):\mathbb{D}\to\hat{\mathbb{C}}
$$
is holomorphic in $w\in\mathbb{D}$ for each $z\in K$ (see
\cite{MSS}). The variable $w\in\mathbb{D}$ is called the {\it
parameter} of the holomorphic motion of $K$. It is also possible to
define holomorphic motions over simply connected regions of
$\mathbb{C}$ when we specify the point where the motion is the
identity.

The lambda lemma states that a holomorphic motion of $K$ extends to
a holomorphic motion of the closure $\bar{K}$ of $K$ (see
\cite{MSS}). Slodkowski \cite{Slo} proved that a holomorphic motion
of a closed set $K$ which contains at least three points extends to
a holomorphic motion of $\hat{\mathbb{C}}$. In fact, if a
holomorphic motion of $K$ is invariant under a subgroup $G$ of
$PSL_2(\mathbb{C})$ then the extension of the holomorphic motion can
be chosen to be $G$-equivariant on $\hat{\mathbb{C}}$ \cite{EKK}.

\vskip .1 cm

A {\it shear-bend transverse cocycle} $\beta$ for a geodesic
lamination $\lambda$ on a closed hyperbolic surface $S$ assigns to
each arc $k$ transverse to $\lambda$ (with endpoints of $k$ in the
plaques of $\lambda$) a number $\beta (k)\in\mathbb{C}/2\pi
i\mathbb{Z}$ such that if $k=k_1\cup k_2$ and $k_1,k_2$ have
disjoint interiors then $\beta (k)=\beta (k_1)+\beta (k_2)$. Denote
by $\mathcal{H}(\lambda ,\mathbb{C}/2\pi\mathbb{Z})$ the space of
all shear-bend transverse cocycles for $\lambda$. Bonahon
\cite{Bon2} proved that the space of all representations of the
fundamental group $\pi_1(S)$ of $S$ in $PSL_2(\mathbb{C})$ which
realize $\lambda$ is homeomorphic to an open subset of
$\mathcal{H}(\lambda ,\mathbb{C}/2\pi i\mathbb{Z})$, where the real
part is restricted to belong to the image of $T(S)$ in
$\mathcal{H}(\lambda ,\mathbb{R})$ and there is no restrictions on
the imaginary part.

\vskip .1 cm

Let $\alpha\in\mathcal{H}(\lambda ,\mathbb{R})$ be in the image of
the Teichm\"uller space $T(S)$. For $w=u+iv\in\mathbb{C}$, we define
$\beta_{w}(k)=(w\alpha (k))\ (\rm{mod}\ 2\pi i\mathbb{Z})$ for each
arc $k$ transverse to $\lambda$. Then
$\beta_{w}\in\mathcal{H}(\lambda ,\mathbb{C}/2\pi i\mathbb{Z})$. Let
$f_{w}:\mathbb{H}^2\to\mathbb{H}^3$ be the shear-bend map
corresponding to $\beta_{\tau}$ as in \cite{Bon2}.

\vskip .1 cm

\begin{theorem}
\label{thm:shear-bend_hol_motion} Let $\alpha\in\mathcal{H}(\lambda
,\mathbb{R})$ be in the image of $T(S)$ and let $f_{(1+w )}$ be the
shear-bend map for $\beta_{(1+w )}\in \mathcal{H}(\lambda
,\mathbb{C}/2\pi i\mathbb{Z})$. Then there exists $r>0$ such that
the shear-bend map
$$f_{(1+w )}:\mathbb{H}^2\to\mathbb{H}^3$$
extends by continuity to a holomorphic motion of
$\partial_{\infty}\mathbb{H}^2$ in $\partial_{\infty}\mathbb{H}^3$
for the parameter $\{w\in\mathbb{C}:|w|<r\}$.
\end{theorem}

\begin{proof}
For $w=u+iv$, consider the hyperbolic surface $S_{1+u}$ obtained by
shearing along the real part of $\beta_{1+w}$ which is
$(1+u)\alpha\in\mathcal{H}(\lambda ,\mathbb{R})$. By \cite{Bon2},
the image of $T(S)$ is a cone in $\mathcal{H}(\lambda ,\mathbb{R})$
and therefore $S_{1+u}$ exists for $r$ small enough. Note that $S_1$
is the original hyperbolic surface $S$.

Let $\tau$ be the train track that carries $\lambda$ used in the
proof of Theorem \ref{thm:main}. We choose $\tau$ such that $w^{*}=\frac{1}{2}\cdot \frac{e^{-2l^{*}}\tanh\frac{l_{*}}{2}}{8\pi}$.  For $|u|$ small
enough, the endpoints of the switches of $\tau$ under the shear map
$f_{(1+u)}$ are close to the switches of $\tau$. By connecting the switches with geodesics for the hyperbolic metric of $(1+u)\alpha$ we
construct a train track $\tau_{1+u}$ which is homotopic to $\tau$. For $|u|$ small enough, we have the constants $l^{*}(\tau_{1+u})$, $l_{*}(\tau_{1+u})$ and $w^{*}(\tau_{1+u})$ are as close as we need to the original constants $l^{*}$, $l_{*}$ and $w^{*}$ of the train track $\tau =\tau_1$.
The constants $w^{*}(\tau_{1+u})$ and $\epsilon (\tau_{1+u})=\delta (\tau_{1+u})$ from the proof of Theorem \ref{thm:main} depend continuously on $l^{*}(\tau_{1+u})$, $l_{*}(\tau_{1+u})$ and $w^{*}(\tau_{1+u})$. Thus they depend continuously on $u$ and are bounded away from $0$ for $u$ small enough.
Then the proof of
Theorem \ref{thm:main} applies to each $\beta_{1+w}$ to
obtain an injective map map for $|w|$ small enough where the bound on $|Im(w)|$ is obtained from Theorem
\ref{thm:main}. It is clear that when $w=0$, we have
$f_{(1+0)}=id$.

Finally, we fix $z\in\partial_{\infty}\mathbb{H}^2$ and consider
$w\mapsto f_{(1+w)}(z)$. Bonahon \cite{Bon2} proved that the
shear-bend map is holomorphic in the transverse cocyle when
restricted to a single plaque of $\tilde{\lambda}'$. Since the
endpoints of the plaques are dense in $\partial_{\infty}\mathbb{H}$,
it follows that $w\mapsto f_{(1+w)}(z)$ is holomorphic in $w$ for
$z$ in a dense subset of $\partial_{\infty}\mathbb{H}$. By the
lambda lemma, this is enough to claim that $f_{(1+w)}$ extends to a
holomorphic motion of $\partial_{\infty}\mathbb{H}^2$ for $w$ in the
described neighborhood of $0\in\mathbb{C}$.
\end{proof}

\end{document}